\documentclass[pre,onecolumn,showpacs,showkeys,superscriptaddress,preprintnumbers, floatfix]{revtex4}

\usepackage{color}
\usepackage{dsfont}
\usepackage{ifpdf}
\usepackage{hyperref}
\usepackage{dcolumn}
\usepackage{url}
\usepackage{amsmath}
\usepackage{amscd}
\usepackage{amsfonts}
\usepackage{amssymb}
\usepackage{bm}   
\usepackage{bbm}   
\usepackage{verbatim}
\usepackage{stmaryrd}
\usepackage{amsthm}

\ifpdf
  \usepackage[pdftex]{graphicx}         
\else
  \usepackage[dvips]{graphicx}
\fi

\newcommand{\eM}     {$\epsilon$\protect\nobreakdash-machine}
\newcommand{\eMs}    {$\epsilon$\protect\nobreakdash-machines}

\newcommand{\EMs}    {$\epsilon$\protect\nobreakdash-Machines}

\newcommand{\Process}{\mathcal{P}}

\newcommand{\MeasSymbol}   { {X} }
\newcommand{\meassymbol}   { {x} }

\newcommand{\biinfinity}	{ \overleftrightarrow {\meassymbol} }

\newcommand{\biinfinitystates}	{ \overleftrightarrow {s} }
\newcommand{\Past}	{ \overleftarrow {\MeasSymbol} }
\newcommand{\past}	{ {\overleftarrow {\meassymbol}} }

\newcommand{\Future}	{ \overrightarrow{\MeasSymbol} }
\newcommand{\future}	{ \overrightarrow{\meassymbol} }

\theoremstyle{plain}   
\theoremstyle{plain}   \newtheorem{Lem}{Lemma}
\theoremstyle{plain} 	\newtheorem{Cor}{Corollary}
\theoremstyle{plain} 	\newtheorem{The}{Theorem}
\theoremstyle{plain} 	
\theoremstyle{plain} 	
\theoremstyle{plain}	\newtheorem*{Rem}{Remark}
\theoremstyle{plain}	\newtheorem{Def}{Definition} 
\theoremstyle{plain}	
\theoremstyle{plain}   \newtheorem{Cla}{Claim}
\theoremstyle{plain} \newtheorem{Exa}{Example}
\theoremstyle{plain} \newtheorem*{Exm}{Example}

\def\norm#1{\|#1\|}

\def\AA{\mathcal{A}}
\def\N{\mathbb{N}}

\def\R{\mathbb{R}}
\def\RR{\mathcal{R}}
\def\TT{\mathcal{T}}
\def\Z{\mathbb{Z}}
\def\X{\mathbb{X}}

\def\FF{\mathcal{F}}
\def\H{\mathbb{H}}
\def\P{\mathbb{P}}
\def\Pt{\widetilde{\mathbb{P}}}
\def\PB{\mathbf{P}}
\def\Ex{\mathbb{E}}
\def\E{\mathcal{E}} 
\def\indicator{\mathds{1}}

\def\XX{\mathcal{X}}
\def\XXZ{\mathcal{X}^{\mathbb{Z}}}
\def\XXm{\mathcal{X}^-}
\def\XcrossS{(\mathbb{X} \mbox{ x } \mathbb{S})}
\def\XSZ{(\mathcal{X} \mbox{ x } \mathcal{S} )^{\mathbb{Z}}}
\def\Pm{\mathbb{P}^-}
\def\Xm{\mathbb{X}^-}

\def\goesonx{\stackrel{x}{\rightarrow}} 
\def\Processt{\widetilde{\mathcal{P}}}

\def\CS{S}
\def\cs{\sigma}
\def\csr{s}
\def\MS{X}
\def\ms{x}
\def\csb{\overline{\sigma}}
\def\csrb{\overline{s}}
\def\CSb{\overline{S}} 
\def\CSSet{\mathcal{S}}

\def\L{\mathcal{L}}

\begin{document}

\title{Equivalence of History and Generator \EMs}

\author{Nicholas F. Travers}
\email{ntravers@math.ucdavis.edu}
\affiliation{Complexity Sciences Center}
\affiliation{Mathematics Department}

\author{James P. Crutchfield}
\email{chaos@ucdavis.edu}
\affiliation{Complexity Sciences Center}
\affiliation{Mathematics Department}
\affiliation{Physics Department\\
University of California at Davis,\\
One Shields Avenue, Davis, CA 95616}
\affiliation{Santa Fe Institute\\
1399 Hyde Park Road, Santa Fe, NM 87501}

\date{\today}

\bibliographystyle{unsrt}


\begin{abstract}

\EMs\ are minimal, unifilar presentations of stationary stochastic processes.
They were originally defined in the history machine sense, as hidden Markov 
models whose states are the equivalence classes of infinite pasts with the same
probability distribution over futures. In analyzing synchronization, though,
an alternative generator definition was given: unifilar, edge-emitting hidden 
Markov models with probabilistically distinct states. The key difference is that
history \eMs\ are defined by a process, whereas  generator \eMs\ define a process. 
We show here that these two definitions are equivalent in the finite-state case. 

\end{abstract}

MCS Codes: 
37A50 
37B10 
60J10 

%
\preprint{Santa Fe Institute Working Paper 11-11-051}
\preprint{arxiv.org:1111.4500 [math.PR]}


\maketitle


\section{Introduction}
\label{sec:Introduction}

Let $\Process = (\MS_t)_{t \in \Z}$ be a stationary stochastic process. The
\eM\ $M = M(\Process)$ for the process $\Process$ is the hidden Markov model
whose states consist of equivalence classes of infinite past sequences
(histories) $\past = \ldots \ms_{-2} \ms_{-1}$ with the same probability
distribution over future sequences $\future = \ms_0 \ms_1 \ldots$. The corresponding
equivalence relation on the set of pasts $\past$ is denoted by $\sim_{\epsilon}$:
\begin{align}
\past \sim_{\epsilon} \past^\prime ~ \mbox{ if } ~
\P(\Future|\past) = \P(\Future|\past^\prime) ~,
\label{eq:PredEquiv}
\end{align}
where $\Future = \MS_0 \MS_1 \ldots$ denotes the infinite sequence of future random variables.  

These machines were first introduced in \cite{Crut88a} as minimal, predictive
models to measure the structural complexity of dynamical systems and have subsequently been applied in a number of contexts for nonlinear modeling
\cite{Crut97a, Varn02a, Varn03c, Stil07b, Li08a}. Important extensions and a
more thorough development of the theory were given in
\cite{Crut92c, Uppe97a, Shal98a, Ay05a}.
However, it was not until quite recently that the first fully formal construction was presented in \cite{Lohr10a}. 

Shortly thereafter, in our studies of synchronization \cite{Trav11a,Trav11b},
we introduced an alternative ``generator \eM'' definition, in contrast to the original ``history \eM''  
construction discussed above. A generator \eM\ is defined simply as a unifilar, edge-emitting hidden Markov
model with probabilistically  distinct states. As opposed to the history \eM\ $M_h = M_h(\Process)$
which is derived from a process $\Process$, a generator \eM\ $M_g$ 
itself defines a stationary process $\Process = \Process(M_g)$. Namely, the stationary output 
process of the hidden Markov model $M_g$ obtained by choosing the initial state according to the
stationary distribution $\pi$ for states of the underlying Markov chain.   

We establish, here, that the history and generator \eM\ definitions are equivalent
in the finite state case. This has long been assumed, without
formally specifying the generator definition. However, our work makes this
explicit and gives one of the first formal proofs of equivalence. 

The equivalence is also implicit in \cite{Lohr10a}; in fact, for a more general
class of machines, not just finite-state. However, the techniques used there
differ substantially from ours and use somewhat more machinery. 
In particular, the proof of equivalence in the more difficult direction of Theorem 
\ref{thm:GeneratorsAreHistoryMachines} (Section \ref{subsec:Generator_eMachines_are_History_eMachines})
uses a supermartingale argument that, though elegant, relies implicitly on the martingale 
convergence theorem and is not particularly concrete. By contrast our proof of Theorem
\ref{thm:GeneratorsAreHistoryMachines} follows directly from the synchronization 
results given in \cite{Trav11a, Trav11b}, which are themselves fairly elementary, using 
only basic information theory and a large deviation estimate for finite-state Markov chains. 
Thus, the alternative proof presented here should be useful in providing intuition for the theorem. 
Also, since the definitions and terminology used in \cite{Lohr10a}
differ significantly from ours, it is not immediately clear that the history-generator 
equivalence is what is shown there or is a consequence of what is shown. Thus,
the exposition here should be helpful in clarifying these issues. 

We note also that in order to parallel the generator \eM\ definition used in our
synchronization studies and apply results from those works, we restrict the
range of processes somewhat when defining history \eMs. In particular, we
assume when defining history \eMs\ that the process $\Process$ is not only
stationary but also ergodic and that the process alphabet is finite. This is
required for equivalence, since the output process of a generator \eM\ is of
this form. However, neither of these assumptions is strictly necessary for
history \eMs. Only stationarity is actually needed. The history \eM\ definition
can be extended to nonergodic stationary processes and countable or even more
general alphabets \cite{Shal98a, Lohr10a}. 

\section{Related Work}
\label{sec:RelatedWork}

Since their introduction in the late 80s, most of the work on \eMs, both
theoretical and applied, has come from the physics and information theory
perspectives. However, similar concepts have been around for some time in
several other disciplines. Among others, there has been substantial work on
related topics by both probabilists and automata theorists, as well as those
in the symbolic dynamics community. Below, we review some of the most germane
developments in these areas. The interested reader is also referred to 
\cite[appendix H]{Shal98a} where a very broad overview of such connections is
given and to \cite{Boyl09a} for a recent review of the relation between symbolic
dynamics and hidden Markov models in general. 

We hope that our review provides context for the study of \eMs\ and
helps elucidate the relationship between \eMs\ and other related models---both
their similarities and their differences. However, an understanding of these
relationships will not be necessary for the equivalence results that follow.
The reader uninterested in these connections may safely skip to the definitions
in Section \ref{sec:Definitions}. 

\subsection{Sofic Shifts and Topological Presentations}
\label{sec:SoficShiftsAndTopologicalPresentations}

Let $\XX$ be a finite alphabet, and let $\XXZ$ denote the set of all
bi-infinite sequences $\biinfinity = \ldots \ms_{-1} \ms_0 \ms_1 \ldots$ consisting
of symbols in $\XX$. A subshift $\Sigma \subset \XXZ$ is said to be
\emph{sofic} if it is the image of a subshift of finite type under a $k$-block
factor map. This concept was first introduced in \cite{Weis73a}, where it also 
shown that any sofic shift $\Sigma$ may be presented as a finite, directed graph
with edges labeled by symbols in the alphabet $\XX$. The allowed
sequences $\biinfinity \in \Sigma$ consist of projections (under the edge
labeling) of bi-infinite walks on the graph edges. 

In the following, we will assume that all vertices in a presenting graph $G$ 
of a sofic shift are \emph{essential}. That is, each vertex $v$ occurs as the target vertex
of some edge $e$ in a bi-infinite walk on the graph edges. If this is not the case, one may 
restrict to the graph $G'$ consisting of essential vertices in $G$, along with their outgoing edges, 
and $G'$ will also be a presenting graph for the sofic shift $\Sigma$. Thus, it is only necessary
to consider presenting graphs in which all vertices are essential. 

The \emph{language} $\L(\Sigma)$ of a subshift $\Sigma$ is the set of all
finite words $w$ occurring in some point $\biinfinity \in \Sigma$. For a sofic
shift $\Sigma$ with presenting graph $G$ one may consider the nondeterministic
finite automaton (NFA) $M$ associated with the graph $G$, in which all states
(vertices of $G$) are both start and accept states. Clearly (under the assumption 
that all vertices are essential) the language accepted by $M$ is just $\L(\Sigma)$.
Thus, the language of any sofic shift is regular. By standard algorithms (see
e.g. \cite{Hopc79}) one may obtain from $M$ a unique, minimal, deterministic
finite automaton (DFA) $M'$ with the fewest number of states of all DFAs
accepting the language $\L(\Sigma)$. We call $M'$ the \emph{minimal
deterministic automaton} for the sofic shift $\Sigma$. 

A subshift $\Sigma$ is said to be \emph{irreducible} if for any two words 
$w_1,w_2 \in \L(\Sigma)$ there exists $w_3 \in \L(\Sigma)$ such that the word 
$w_1 w_3 w_2 \in \L(\Sigma)$. As shown in \cite{Fisc75a}, a sofic shift
is irreducible if and only if it has some irreducible (i.e. strongly connected)
presenting graph $G$. 

A presenting graph $G$ of a sofic shift $\Sigma$ is said to be \emph{unifilar}
or \emph{right-resolving} if for each vertex $v \in G$ and symbol $x \in \XX$,
there is at most one outgoing edge $e$ from $v$ labeled with the symbol $x$.
A shown in \cite{Boyl85a}, an irreducible sofic shift $\Sigma$ always has a unique, 
minimal, unifilar presenting graph, that, it turns out, is also irreducible. 
In symbolic dynamics this presentation is often referred to as the (right) \emph{Fischer cover} of $\Sigma$.  

For an irreducible sofic shift $\Sigma$, the graph associated with the minimal
deterministic automaton always has a single recurrent, irreducible component.
This recurrent component is \emph{isomorphic} to the Fischer cover. That is,
there exists a bijection between vertices in the automaton graph and vertices
of the Fischer cover that preserves both edges and edge labels.

A related notion is the Krieger cover based on future sets \cite{Lind95a}. For a
subshift $\Sigma \subset \XXZ$, let $\Sigma^+$ denote  the set of allowed future
sequences $\future = \ms_0 \ms_1 \ldots$ and let $\Sigma^-$ be the set of
allowed past sequences $\past = \ldots \ms_{-2} \ms_{-1}$. That is: 
\begin{align*}
\Sigma^+ = \{\future: \exists \past \mbox{ with } \past \future \in \Sigma\} \mbox{ and }
\Sigma^- = \{\past: \exists \future \mbox{ with } \past \future \in \Sigma\} .
\end{align*} 
Also, for a past $\past \in \Sigma^-$, let the \emph{future set} $F(\past)$ of
$\past$ be the set of all possible future sequences $\future$ that can follow
$\past$: 
\begin{align*}
F(\past) = \{\future \in \Sigma^+: \past \future \in \Sigma\} . 
\end{align*}
Define an equivalence relation $\sim_K$ on the set of infinite pasts $\past \in \Sigma^-$ by:
\begin{align}
\past \sim_K \past' \mbox{ if } F(\past) = F(\past').
\label{eq:tildeK}
\end{align}
The Krieger cover of $\Sigma$ is the (possibly infinite) directed, edge-labeled 
graph $G$ whose vertices consist of equivalence classes of pasts $\past$ under
the relation $\sim_K$. There is a directed edge in $G$ from vertex $v$ to vertex
$v'$ labeled with symbol $\ms$, if for some past $\past \in v$
(equivalently all pasts $\past \in v$) the past $\past' = \past \ms \in v'$.
By construction, the Krieger cover $G$ is necessarily unifilar. Moreover, it is
easily shown that $G$ is a finite graph if and only if the subshift $\Sigma$ is sofic. 

If the subshift $\Sigma$ is both irreducible and sofic, then the Krieger cover
is isomorphic to the subgraph of the minimal deterministic automaton
consisting of all states $v$ that are not $\emph{finite-time transient}$ (and
their outgoing edges). That is, the subgraph consisting of those states $v$ such
that there exist arbitrarily long words $w \in \L(\Sigma)$ on which the
automaton transitions from its start state to $v$. Clearly, any state in the
recurrent, irreducible component of  the automaton graph is not finite-time
transient. Thus, the Krieger cover contains this recurrent component---the Fischer cover.

To summarize, the minimal deterministic automaton, Fischer cover, and Krieger
cover are three closely related ways for presenting an irreducible sofic shift
that are each, in slightly different senses, minimal unifilar presentations.
The Fischer cover is always an irreducible graph. The Krieger cover and graph
of the minimal deterministic automaton are not necessarily irreducible, but
they each have a single recurrent, irreducible component that is isomorphic to
the Fischer cover. The Krieger cover itself is also isomorphic to a subgraph
of the minimal deterministic automaton. 

\EMs\ are a probabilistic extension of these purely topological presentations.
More specifically, for a stationary process $\Process$ the history \eM\ $M_h$
is the probabilistic analog of the Krieger cover $G$ for the subshift consisting
of $\mbox{supp}(\Process)$. It is the edge-emitting hidden Markov model defined
analogously to the Krieger cover, but with states that are equivalence classes
of infinite past sequences $\past$ with the same probability distribution
over future sequences $\future$, rather than simply the same set of allowed
future sequences. (Compare Equations (\ref{eq:PredEquiv}) and (\ref{eq:tildeK}).)

In some cases the two presentations may be topologically equivalent---e.g., the
history \eM\ and Krieger cover can be isomorphic as graphs when
the transition probabilities are removed from edges of the \eM.
In other cases, however, they are not. For example, for the Even Process
(Example \ref{ex:Even}, Section \ref{subsec:Examples}) the Krieger cover (or at least
its recurrent component, the Fischer cover) and the history \eM\ are topologically 
equivalent. But this is not so for the ABC process (Example \ref{ex:ABC},
Section \ref{subsec:Examples}). In fact, there exist many examples of ergodic
processes whose support is an irreducible sofic shift, but for which the
history \eM\ has an infinite (or even continuum) number of states. See, e.g.,
Example \ref{ex:SNS} in Section \ref{subsec:Examples}, and Example 3.26 in \cite{Lohr10a}.

\subsection{Semigroup Measures}
\label{sec:SemigroupMeasures}

Semigroup measures are a class of probability measures on sofic shifts that
arise from assigning probability transition structures to the right and left
covers obtained from the Cayley graphs associated with generating semigroups
for the shifts. These measures are studied extensively in \cite{Kitc84a}, where 
a rich theory is developed and many of their key structural properties are characterized.

In particular, it is shown there that a stationary probability measure $\P$ on 
a sofic shift $\Sigma$ is a semigroup measure if and only if it has a finite number of
\emph{future measures}---distributions over future sequences
$\future$---induced by all finite-length past words $w$.
That is, if there exist a finite number of finite length words
$w_1, \ldots , w_N$ such that for any word $w$ of positive probability:
\begin{align*}
\P(\Future | w) = \P(\Future | w_i) ~,
\end{align*}
for some $1 \leq i \leq N$, where $\Future$ denotes the
infinite sequence of future random variables $\MS_t$ on $\Sigma$, 
defined by the natural projections $\MS_t(\biinfinity) = \ms_t$. 

By contrast, a process $\Process$ (or measure $\P$) has a finite-state history
\eM\ if there exist a finite number of infinite past sequences
$\past_1, \ldots, \past_N$ such that, for almost every infinite past $\past$:
\begin{align*}
\P(\Future|\past) = \P(\Future| \past_i) ~,
\end{align*}
for some $1 \leq i \leq N$. The latter condition is strictly more general. The
Alternating Biased Coins Process described in Section \ref{subsec:Examples}, for
instance, has a finite-state ($2$-state) \eM, but does not correspond to
a semigroup measure.

Thus, unfortunately, though the theory of semigroup measures is quite rich and well
developed, much of it does not apply for the measures we study. For this reason, 
our proof methods are quite different from those previously used for semigroup 
measures, despite the seeming similarity between the two settings

\subsection{g-Functions and g-Measures}
\label{sec:gfunctions_and_gmeasures}

For a finite alphabet $\XX$, let $\XX^-$ denote the set of all infinite
past sequences $\past = \ldots \ms_{-2} \ms_{-1}$ consisting of symbols in $\XX$. 
A \emph{g-function} for the full shift $\XX^{\Z}$ is a map:
\begin{align*}
g: (\XX^- \times \XX) \rightarrow [0,1] ~,
\end{align*}
such that for any $\past \in \XX^-$: 
\begin{align*}
\sum_{x \in \XX} g(\past, x) = 1 ~.
\end{align*}
A \emph{g-measure} for a g-function $g$ on the full shift $\XXZ$ is stationary
probability measure $\P$ on $\XX^{\Z}$ that is consistent with the g-function
$g$ in that for $\P$ a.e. $\past \in \XX^-$: 
\begin{align*}
\P(\MS_0 = x|\Past = \past) = g(\past, x), ~ \mbox{for each $x \in \XX$}. 
\end{align*}
g-Functions and g-Measures have been studied for some time, though sometimes
under different names \cite{Harr55a,Kean72a,Bram93a,Sten03a}. In particular,
many of these studies address when a g-function will or will not have a unique
corresponding g-measure. Normally, $g$ is assumed to be continuous (with
respect to the natural product topology) and in this case, using fixed point theory, it can
be shown that at least one g-measure exists. However, continuity is not enough
to ensure uniqueness, even if some natural mixing conditions are required as
well \cite{Bram93a}. Thus, stronger conditions are often required, such as
H\"{o}lder continuity.

Of particular relevance to us is the more recent work \cite{Krie10a} on 
g-functions restricted to subshifts. It is shown there, in many instances, 
how to construct g-functions on subshifts with an infinite or even continuum
number of future measures, subject to fairly strong requirements. For example, 
residual local constancy or a synchronization condition
similar to the exactness condition introduced in \cite{Trav11a}. Most surprising, 
perhaps, are the constructions of g-functions for irreducible subshifts,
which themselves take only a finite number of values, but have unique associated 
g-measures with an infinite number of future measures. 
 
The relation to \eMs\ is the following. Given a g-function $g$,
one may divide the set of infinite past sequences $\past$ into 
equivalence classes, in a manner analogous to that for history \eMs,
by the relation $\sim_g$:
\begin{align}
\past \sim_g \past' \mbox{ if } g(\past, \ms) = g(\past', \ms),
	\mbox{ for all } \ms \in \XX ~.
\label{eq:tildeg}
\end{align}
The equivalence classes induced by the relation $\sim_g$ of Equation \ref{eq:tildeg}  
are coarser than those induced by the relation $\sim_{\epsilon}$ of Equation \ref{eq:PredEquiv}. 
For any g-measure $\P$ of the g-function $g$, the states of the history \eM\
are a refinement or splitting of the $\sim_g$ equivalence classes. Two infinite pasts
$\past$ and $\past'$ that induce different probability distributions over the
next symbol $\ms_0$ must induce different probability distributions over
infinite future sequences $\future$, but the converse is not necessarily 
true. As shown in \cite{Krie10a}, the splitting may, in fact,
be quite ``bad'' even if ``nice'' conditions are enforced on the g-function
associated with the probability measure $\P$. Concretely, there exist processes
with history \eMs\ that have an infinite or even continuum number of states,
but for which the associated ``nice'' g-function from which the process is derived
has only a finite number of equivalence classes.  

\section{Definitions}
\label{sec:Definitions}

In this section we set up the formal framework for our results and give more complete definitions
for our objects of study: stationary processes, hidden Markov models, and \eMs. 

\subsection{Processes}
\label{subsec:Processes}

There are several ways to define a stochastic process. Perhaps the most
traditional is simply as a sequence of random variables $(\MS_t)$ on some
common probability space $\Omega$. However, in the following it will be
convenient to use a slightly different, but equivalent, construction in which
a process is itself a probability space whose sample space consists of
bi-infinite sequences $\biinfinity = \ldots \ms_{-1} \ms_0 \ms_1 \ldots$. Of
course, on this space we have random variables $\MS_t$ defined by the natural
projections $\MS_t(\biinfinity) = \ms_t$, which we will employ at times in our
proofs. However, for most of our development and, in particular, for defining
history \eMs, it will be more convenient to adopt the sequence-space viewpoint.  

Throughout, we restrict our attention to processes over a finite alphabet
$\XX$. We denote by $\XX^*$ the set of all words $w$ of finite positive 
length consisting of symbols in $\XX$ and, for a word $w \in \XX^*$, we write
$|w|$ for its length. Note that we deviate slightly from the standard convention
here and explicitly exclude the null word $\lambda$ from $\XX^*$.

\begin{Def}
\label{def:Process}
Let $\XX$ be a finite set. A process $\Process$ over the alphabet $\XX$ is
a probability space $(\XXZ, \X, \P)$ where:
\begin{itemize}
\item $\XXZ$ is the set of all bi-infinite sequences of symbols in $\XX$: 
	$\XXZ = \{\biinfinity = \ldots \ms_{-1} \ms_0 \ms_1 \ldots : \ms_t \in \XX,
	\mathrm{~for~all~} t \in \Z\}$.
\item $\X$ is the $\sigma$-algebra generated by finite cylinder sets of the
	form $A_{w,t} = \{\biinfinity \in \XXZ : \ms_{t} \ldots \ms_{t+|w|-1} = w\}$.
\item $\P$ is a probability measure on the measurable space $(\XXZ, \X)$. 
\end{itemize}
\end{Def}

For each symbol $\ms \in \XX$, we assume implicitly that $\P(A_{\ms,t}) > 0$ for
some $t \in \N$. Otherwise, the symbol $\ms$ is useless and the process can be
restricted to the alphabet $\XX/\{\ms\}$. In the following, we will be
primarily interested in stationary, ergodic processes.

Let $r: \XXZ \rightarrow \XXZ$ be the right shift operator. A process
$\Process$ is \emph{stationary} if the measure $\P$ is shift invariant:
$\P(A) = \P(r(A))$ for any measurable set $A$. A process $\Process$ is
\emph{ergodic} if every shift invariant event $A$ is trivial. That is, for any
measurable event $A$ such that $A$ and $r(A)$ are $\P$ a.s. equal, the
probability of $A$  is either $0$ or $1$. A stationary process $\Process$ is
defined entirely by the word probabilities $\P(w)$, $w \in \XX^*$, where
$\P(w) = \P(A_{w,t})$ is the shift invariant probability of cylinder sets for
the word $w$. Ergodicity is equivalent to the almost sure convergence of
empirical word probabilities $\widehat{\P}(w)$ in finite sequences
$\future^t = \ms_0 \ms_1 \ldots \ms_{t-1}$ to their true values $\P(w)$,
as $t \rightarrow \infty$. 

For a stationary process $\Process$ and words $w,v \in \XX^*$ with $\P(v) > 0$,
we define $\P(w|v)$ as the probability that the word $w$ is followed by the
word $v$ in a bi-infinite sequence $\biinfinity$:
\begin{align}
\label{eq:conditional_word_prob}
\P(w|v) 	& \equiv \P(A_{w,0}|A_{v,-|v|}) \nonumber \\
	 	& = \P(A_{v,-|v|} \cap A_{w,0})/\P(A_{v,-|v|}) \nonumber \\
		& = \P(vw)/\P(v) ~.
\end{align}
The following facts concerning word probabilities and conditional word
probabilities for a stationary process come immediately from the definitions.
They will be used repeatedly throughout our development, without further
mention. For any words $u,v,w \in \XX^*$: 
\begin{enumerate}
\item $\sum_{\ms \in \XX} \P(w \ms) = \sum_{\ms \in \XX} \P(\ms w) = \P(w)$;
\item $\P(w) \geq \P(wv)$ and $\P(w) \geq \P(vw)$;
\item If $\P(w) > 0$, $\sum_{\ms \in \XX} \P(\ms|w) = 1$;
\item If $\P(u) > 0$, $\P(v|u) \geq \P(vw|u)$; and
\item If $\P(u) > 0$ and $\P(uv) > 0$, $\P(vw|u) = \P(v|u) \cdot \P(w|uv)$.
\end{enumerate}

\subsection{Hidden Markov Models}
\label{subsec:HiddenMarkovModels}

There are two primary types of hidden Markov models: \emph{state-emitting} (or
\emph{Moore}) and \emph{edge-emitting} (or \emph{Mealy}). The state-emitting
type is the simpler of the two and, also, the more commonly studied and applied
\cite{Rabi89a,Ephr02a}. However, we focus on edge-emitting hidden Markov models
here, since \eMs\ are edge-emitting. We also restrict to the case where the hidden
Markov model has a finite number of states and output symbols, although
generalizations to countably infinite and even uncountable state sets and
output alphabets are certainly possible. 

\begin{Def}
\label{def:HMM}
An edge-emitting \emph{hidden Markov model} \emph{(HMM)} is a triple 
$(\CSSet, \XX, \{T^{(\ms)}\})$ where:
\begin{itemize}
\item $\CSSet $ is a finite set of states,
\item $\XX$ is a finite alphabet of output symbols, and
\item $T^{(\ms)}, \ms \in \XX$ are symbol-labeled transition matrices.
	$T_{\cs \cs'}^{(\ms)} \geq 0$ represents the probability of transitioning from
	state $\cs$ to state $\cs'$ on symbol $\ms$.
\end{itemize}
\end{Def}
In what follows, we normally take the state set to be
$\CSSet = \{\cs_1, \ldots , \cs_N\}$ and denote $T_{\cs_i \cs_j}^{(\ms)}$ simply
as $T_{ij}^{(\ms)}$. We also denote the overall state-to-state transition matrix
for an HMM as $T$: $T = \sum_{\ms \in \XX} T^{(\ms)}$. $T_{ij}$ is the overall
probability of transitioning from state $\cs_i$ to state $\cs_j$, regardless of
symbol. The matrix $T$ is stochastic: $\sum_{j=1}^N T_{ij} = 1$, for each $i$. 

Pictorially, an HMM can be represented as a directed graph with labeled
edges. The vertices are the states $\cs_1, \ldots, \cs_N$ and, for each
$i,j,\ms$ with $T^{(\ms)}_{ij} > 0$, there is a directed edge from state $\cs_i$
to state $\cs_j$ labeled $p|\ms$ for the symbol $\ms$ and transition probability
$p = T^{(\ms)}_{ij}$. The transition probabilities are normalized so that their
sum on all outgoing edges from each state $\cs_i$ is $1$. 

\begin{Exm}
Even Machine
\label{ex:EvenGenerator}
\end{Exm}

Figure \ref{fig:EvenProcess} depicts an HMM for the Even Process. The support
for this process consists of all binary sequences in which blocks of uninterrupted 
$1$s are even in length, bounded by $0$s. After each even length is reached, there 
is a probability $p$ of breaking the block of $1$s by inserting a $0$. The HMM has two
states $\{\cs_1,\cs_2\}$ and symbol-labeled transitions matrices:
\begin{align*}
T^{(\textcolor{blue}{0})} =
  \left(
  \begin{array}{cc}
  	p & 0 \\
	0 & 0 \\
  \end{array}
  \right)
\mathrm{~~and~~}
T^{(\textcolor{blue}{1})} =
  \left(
  \begin{array}{cc}
	0 & 1-p \\
	1 & 0 \\
  \end{array}
  \right) 
\end{align*}

\begin{figure}[ht]
\begin{minipage}[b]{0.4\textwidth}
\centering
\includegraphics[scale=0.9]{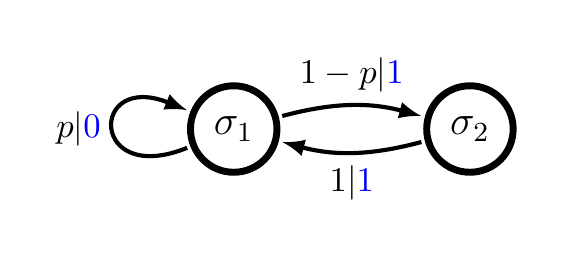}
\end{minipage}
\caption{A hidden Markov model (the \eM) for the Even Process. The HMM has two
  internal states $\CSSet = \{ \cs_1, \cs_2 \}$, a two-symbol alphabet
  $\XX = \{0,1\}$, and a single parameter $p \in (0,1)$ that controls the
  transition probabilities. }
\label{fig:EvenProcess}
\end{figure}

The operation of an HMM may be thought of as a weighted random walk on the 
associated directed graph. That is, from the current state $\cs_i$, the next
state $\cs_j$ is determined by selecting an outgoing edge from $\cs_i$
according to their relative probabilities. Having selected a transition, the
HMM then moves to the new state and outputs the symbol $x$ labeling this 
edge. The same procedure is then invoked repeatedly to generate future
states and output symbols.  

The state sequence determined in such a fashion is simply a Markov chain with
transition matrix $T$. However, we are interested not simply in the HMM's state
sequence, but rather the associated sequence of output symbols it generates. 
We assume that an observer of the HMM has direct access to this sequence 
of output symbols, but not to the associated sequence of ``hidden'' states. 

Formally, from an initial state $\cs_i$ the probability that the HMM next 
outputs symbol $x$ and transitions to state $\cs_j$ is:
\begin{align}
\PB_{\cs_i}(\ms,\cs_j) = T^{(\ms)}_{ij} ~.
\end{align}
And, the probability of longer sequences is computed inductively. 
Thus, for an initial state $\cs_i = \cs_{i_0}$ the probability the HMM outputs
a length-$l$ word $w = w_0 \ldots w_{l-1}$ while following the \emph{state path} 
$\csr = \cs_{i_1} \ldots \cs_{i_l}$ in the next $l$ steps is:   
\begin{align}
\PB_{\cs_i}(w,s)= \prod_{t=0}^{l-1} T^{(w_t)}_{i_t,i_{t+1}}  ~.
\end{align}
If the initial state is chosen according to some distribution
$\rho = (\rho_1, \ldots, \rho_N)$ 
rather than as a fixed state $\sigma_i$, we have by linearity:
\begin{align}
\PB_{\rho}(\ms, \cs_j) & = \sum_i \rho_i \cdot \PB_{\cs_i}(\ms, \cs_j)
	\mbox{ and } \\
\PB_{\rho}(w,s) & =  \sum_i \rho_i \cdot  \PB_{\cs_i}(w,s) ~.
\label{eq:symbol-state_cylinder_probs}
\end{align}
The overall probabilities of next generating a symbol $\ms$ or word $w = w_0
\ldots w_{l-1}$ from a given state 
$\cs_i$ are computed by summing over all possible associated target states or state sequences:
\begin{align}
\PB_{\cs_i}(\ms) & = \sum_j \PB_{\cs_i}(\ms,\cs_j) 
   = \norm{e_i T^{(\ms)}}_1
	\mbox{ and }\\
\PB_{\cs_i}(w) & = \sum_{\{s:|s|=l\}} \PB_{\cs_i}(w,s) 
    = \norm{e_i T^{(w)}}_1 ~, 
\label{eq:StateToWordProb}
\end{align}
respectively, where $e_i = (0,\ldots,1,\ldots,0)$ is the $i^\text{th}$ standard
basis vector in $\R^N$ and
\begin{align}
T^{(w)} & = T^{(w_0 \ldots w_{l-1})} 
         \equiv \prod_{t=0}^{l-1} T^{(w_t)} ~.
\end{align}
Finally, the overall probabilities of next generating a symbol $\ms$ or word
$w = w_0 \ldots w_{l-1}$ from an initial state distribution $\rho$ are, respectively:
\begin{align}
\PB_{\rho}(\ms) & = \sum_i \rho_i \cdot \PB_{\cs_i}(\ms) 
= \norm{\rho T^{(\ms)}}_1  \mbox{ and } \\
\PB_{\rho}(w) & = \sum_i \rho_i \cdot \PB_{\cs_i}(w)
= \norm{\rho T^{(w)}}_1 ~.
\end{align}

If the graph $G$ associated with a given HMM is strongly connected, then the 
corresponding Markov chain over states is irreducible and the state-to-state 
transition matrix $T$ has a unique \emph{stationary distribution} $\pi$ satisfying
$\pi = \pi T$ \cite{Levi06}. In this case, we may define a stationary process 
$\Process = (\XXZ, \X, \P)$ by the word probabilities obtained from choosing
the initial state according to $\pi$. That is, for any word $w \in \XX^*$:
\begin{align}
\P(w) & \equiv \PB_{\pi}(w) = \norm{\pi T^{(w)}}_1 ~.
\label{eq:generator_word_probs}
\end{align}
Strong connectivity also implies the process $\Process$ is ergodic, as it is
a pointwise function of the irreducible Markov chain over edges, which
is itself ergodic \cite{Levi06}. That is, at each time step the symbol
labeling the edge is a deterministic function of the edge.

We denote the corresponding (stationary, ergodic) process over bi-infinite
symbol-state sequences $(\biinfinity, \biinfinitystates)$ by $\Processt$.
That is, $\Processt = (\XSZ, \XcrossS, \Pt)$ where: 
\begin{enumerate}
\item $\XSZ = \left\{ (\biinfinity,\biinfinitystates) \cong (\ms_t,s_t)_{t \in \Z} : 
	\ms_t \in \XX \mbox{ and } \csr_t \in \CSSet, \mbox{ for all } t \in \Z \right\}$. 
\item $\XcrossS$ is the $\sigma$-algebra generated by finite cylinder sets on
	the bi-infinite symbol-state sequences. 
\item The (stationary) probability measure $\Pt$ on $\XcrossS$ is defined by
	Equation (\ref{eq:symbol-state_cylinder_probs}) with $\rho = \pi$. Specifically,
	for any length-$l$ word $w$ and length-$l$ state sequence $s$ we have: 
	\begin{align*}
	\Pt( \{ (\biinfinity,\biinfinitystates):\ms_0 \ldots \ms_{l-1} = w,
		\csr_1 \ldots \csr_l = s \}) = \PB_{\pi} (w, s).
	\end{align*}
	By stationarity, this measure may be extended uniquely to all finite
	cylinders and, hence, to all $\XcrossS$-measurable sets. And, it is
	consistent with the measure $\P$ in that:
	\begin{align*}
	\Pt( \{ (\biinfinity,\biinfinitystates):\ms_0 \ldots \ms_{l-1} = w\})
		= \P(w) ~,
	\end{align*}
	for all $w \in \XX^*$.
\end{enumerate}

Two HMMs are said to be \emph{isomorphic} if there is a bijection between
their state sets that preserves edges, including the symbols and probabilities
labeling the edges. Clearly, any two isomorphic, irreducible HMMs generate the
same process, but the converse is not true. Nonisomorphic HMMs may also
generate equivalent processes. In Section \ref{sec:Equivalence} we will be
concerned with isomorphism between generator and history \eMs. 

\subsection{Generator \EMs}
\label{subsec:Generator_eMachines}

Generator \eMs\ are irreducible HMMs with two additional important properties: 
unifilarity and probabilistically distinct states.

\begin{Def}
\label{def:GeM}
A \emph{generator \eM} $M_g$ is an HMM with the following properties:
\begin{enumerate}
\item \emph{Irreducibility}: The graph $G$ associated with the HMM is strongly connected.
\item \emph{Unifilarity}: For each state $\cs_i \in \CSSet$
	and each symbol $\ms \in \XX$ there is at most one
	outgoing edge from state $\cs_i$ labeled with symbol $\ms$.
\item \emph{Probabilistically distinct states}: For each pair of distinct states
	$\cs_i, \cs_j \in \CSSet$ there exists some word $w \in \XX^*$ 
	such that $\PB_{\cs_i}(w) \not= \PB_{\cs_j}(w)$. 
\end{enumerate} 
\end{Def}

Note that all three of these properties may be easily checked for a given HMM. 
Irreducibility and unifilarity are immediate. The probabilistically distinct
states condition can (if necessary) be checked by inductively separating 
distinct pairs with an algorithm similar to the one used to check for 
topologically distinct states in \cite{Trav11a}.

By irreducibility, there is always a unique stationary distribution $\pi$ over the
states of a generator \eM, so we may associate to each generator \eM\ $M_g$ a 
unique stationary, ergodic process $\Process = \Process(M_g)$ with word 
probabilities defined as in Equation (\ref{eq:generator_word_probs}). We refer to
$\Process$ as \emph{the process} generated by the generator \eM\ $M_g$. 
The transition function for a generator \eM\ or, more generally, any unifilar
HMM is denoted by $\delta$. That is, for $i$ and $\ms$ with
$\PB_{\cs_i}(\ms) > 0$, $\delta(\cs_i,\ms)  \equiv \cs_j$, where $\cs_j$ is
the (unique) state to which state $\cs_i$ transitions on symbol $\ms$. 

In a unifilar HMM, for any given initial state $\sigma_i$ and word
$w = w_0 \ldots w_{l-1} \in \XX^*$, there can be at most one associated state
path $s = s_1 \ldots s_l$ such that the word $w$ may be generated following the
state path $s$ from $\sigma_i$.  Moreover, the probability $\PB_{\cs_i}(w)$ of
generating $w$ from $\cs_i$ is nonzero if and only if there is such a path $s$.
In this case, the states $s_1,\ldots,s_l$ are defined inductively by the
relations $s_{t+1} = \delta(s_t,w_t), 0 \leq t \leq l-1$ with $s_0 = \cs_i$,
and the probability $\PB_{\cs_i}(w)$ is simply:
\begin{align}
\label{eq:WordProbsUnifilarHMM}
\PB_{\cs_i}(w) = \prod_{t=0}^{l-1} \PB_{s_t}(w_t).
\end{align}
Slightly more generally, Equation (\ref{eq:WordProbsUnifilarHMM}) holds as long
as there is a well defined path $s_1 \ldots s_{l-1}$ upon which the subword 
$w_0 \ldots w_{l-2}$ may be generated starting in $\cs_i$. Though, in this case
$\PB_{\cs_i}(w)$ may be $0$ if state $s_{l-1}$ has no outgoing transition on
symbol $w_{l-1}$. This formula for word probabilities in unifilar HMMs
will be useful in establishing the equivalence of generator and history
\eMs\ in Section \ref{sec:Equivalence}. 

\subsection{History \EMs}
\label{subsec:History_eMachines}

The history \eM\ $M_h$ for a stationary process $\Process$ is, essentially, just
the hidden Markov model whose states are the equivalence classes of
infinite past sequences defined by the equivalence relation $\sim_{\epsilon}$ 
of Equation (\ref{eq:PredEquiv}). Two pasts $\past$ and $\past'$ are considered equivalent 
if they induce the same probability distribution over future sequences. 
However, it takes some effort to make this notion precise and specify the transitions.
The formal definition itself is quite lengthy, so for clarity verification of many technicalities
is deferred to the appendices. We recommend first reading through this section in its entirety 
without reference to the appendices for an overview and, then, reading through the
appendices separately afterward for the details. The appendices are entirely
self contained in that, except for the notation introduced here, none of the
results derived in the appendices relies on the development in this section.
As noted before, our focus is restricted to ergodic, finite-alphabet processes
to parallel the generator definition. Although, neither of these requirements
is strictly necessary. Only stationarity is actually needed.

Let $\Process = (\XXZ,\X,\P)$ be a stationary, ergodic process over a finite
alphabet $\XX$, and let $(\XXm,\Xm,\Pm)$ be the corresponding probability
space over past sequences $\past$. That is:
\begin{itemize}
\item $\XXm$ is the set of infinite past sequences of symbols in $\XX$: 
	$\XXm = \{\past = \ldots \ms_{-2}\ms_{-1} : \ms_t \in \XX, t = -1, -2, \ldots \}$.
\item $\Xm$ is the $\sigma$-algebra generated by finite cylinder sets on past sequences:
	$\Xm = \sigma \left(\bigcup_{t=1}^{\infty} \Xm_t \right)$, 
	where $\Xm_t = \sigma \left(\{A_w^-: |w| = t \} \right)$
	and $A_w^- = \{\past = \ldots \ms_{-2}\ms_{-1} : \ms_{-|w|} \ldots \ms_{-1} = w\}$. 
\item $\Pm$ is the probability measure on the measurable space $(\XXm, \Xm)$
	 which is the projection of $\P$ to past sequences:
	$\Pm(A_w^-) = \P(w)$ for each $w \in \XX^*$.  
\end{itemize}

For a given past $\past \in \XXm$, we denote the last $t$ symbols of $\past$
as $\past^t =\ms_{-t} \ldots \ms_{-1}$. A past $\past \in \XXm$ is said to be
\emph{trivial} if $\P(\past^t) = 0$ for some finite $t$ and \emph{nontrivial}
otherwise. If a past $\past$ is nontrivial, then for each $w \in \XX^*$
$\P(w|\past^t)$ is well defined for each $t$, Equation (\ref{eq:conditional_word_prob}), 
and one may consider $\lim_{t \to \infty} \P(w|\past^t)$. A nontrivial past $\past$ is said 
to be \emph{w-regular} if $\lim_{t \to \infty} \P(w|\past^t)$ exists and
\emph{regular} if it is $w$-regular for each $w \in \XX^*$. Appendix
\ref{appsec:Regular_Pasts_and_Trivial_Pasts} shows that the set of trivial
pasts $\TT$ is a null set and that the set of regular pasts $\RR$ has full
measure. That is, $\Pm(\TT) = 0$ and $\Pm(\RR) = 1$. 

For a word $w \in \XX^*$ the function $\PB(w|\cdot) : \RR \rightarrow \R$ is defined by:
\begin{align}
\PB(w|\past) \equiv \lim_{t \to \infty} \P(w|\past^t) ~. 
\end{align}
Intuitively, $\PB(w|\past)$ is the conditional probability of $w$ given $\past$.
However, this probability is technically not well defined in the sense of Equation
(\ref{eq:conditional_word_prob}), since the probability of each past $\past$
is normally $0$. And, we do not want to define $\PB(w|\past)$ in terms of a formal
conditional expectation, because such a definition is only unique up to a.e.
equivalence, while we would like its value on individual pasts to be uniquely
determined. Nevertheless, intuitively speaking, $\PB(w|\past)$ is the
conditional probability of $w$ given $\past$, and this intuition should be kept
in mind as it will provide understanding for what follows. Indeed, 
if one does consider the conditional probability $\P(w|\Past)$ as a
formal conditional expectation, any version of it will be equal to
$\PB(w|\past)$ for a.e. $\past$. So, this intuition is justified.

The central idea in the construction of the history \eM\ is the
following equivalence relation on the set of regular pasts:
\begin{align}
\past \sim \past'  \mbox{ if } \PB(w|\past) = \PB(w|\past') ~,
  ~ \mathrm{~for~all~} w \in \XX^* ~.
\end{align}
That is, two pasts $\past$ and $\past'$ are $\sim$ equivalent if their predictions are the same: 
Conditioning on either past leads to the same probability distribution over future words of all lengths. 
This is simply a more precise definition of the equivalence relation $\sim_{\epsilon}$ of 
Equation (\ref{eq:PredEquiv}). (We drop the subscript $\epsilon$, as this is the only 
equivalence relation we will consider from here on.)

The set of equivalence classes of regular pasts under the relation $\sim$ is denoted as
$\E = \{E_{\beta}, \beta \in B\}$, where $B$ is simply an index set. 
In general, there may be finitely many, countably many, or uncountably
many such equivalence classes. Examples with $\E$ finite and countably infinite
are given in Section \ref{subsec:Examples}. For uncountable $\E$, see example
3.26 in \cite{Lohr10a}.

For an equivalence class $E_{\beta} \in \E$ and word $w \in \XX^*$ we define
the probability of $w$ given $E_{\beta}$ as:
\begin{align}
\PB(w|E_{\beta}) \equiv \PB(w|\past)~, \past \in E_{\beta}.
\end{align}
By construction of the equivalence classes this definition is independent of
the representative $\past \in E_{\beta}$, and Appendix
\ref{appsec:Well_Definedness_of_Equivalence_Class_Transitions} shows that these
probabilities are normalized, so that for each equivalence class $E_{\beta}$:
\begin{align}
\sum_{\ms \in \XX} \PB(\ms|E_{\beta}) = 1 ~.
\label{eq:Normalization}
\end{align}
Appendix \ref{appsec:Well_Definedness_of_Equivalence_Class_Transitions} also
shows that the equivalence-class-to-equivalence-class transitions for the
relation $\sim$ are well defined in that:
\begin{enumerate}
\item For any regular past $\past$ and symbol $\ms \in \XX$ with
	$\PB(\ms|\past) > 0$, the past $\past \ms$ is also a regular.
\item If $\past$ and $\past'$ are two regular pasts in the same equivalence
	class $E_{\beta}$ and $\PB(\ms|E_{\beta}) > 0$, then the two pasts
	$\past \ms$ and $\past' \ms$ must also be in the same equivalence class.
\end{enumerate}
So, for each $E_{\beta} \in \E$ and $\ms \in \XX$ with $\PB(\ms|E_{\beta}) > 0$ there 
is a unique equivalence class $E_{\alpha} = \delta_h(E_{\beta},\ms)$ to which 
equivalence class $E_{\beta}$ transitions on symbol $\ms$. 
\begin{align}
\label{eq:DeltaFunctionHistoryMachineDef} 
\delta_h(E_{\beta},\ms) \equiv E_{\alpha}, \mbox{ where } \past \ms \in E_{\alpha} \mbox{ for } \past \in E_{\beta}. 
\end{align}
By point 2 above, this definition is again independent of the representative 
$\past \in E_\beta$. 

The subscript $h$ in $\delta_h$ indicates that it is a transition
function between equivalence classes of pasts, or histories, $\past$. Formally,
it is to be distinguished from the transition function $\delta$ between the
states of a unifilar HMM. However, the two are essentially equivalent for a
history \eM.

Appendix \ref{appsec:Measurability_of_Equivalence_Classes} shows that each
equivalence class $E_{\beta}$ is an $\Xm$ measurable set, so we can
meaningfully assign a probability:
\begin{align}
\P(E_{\beta}) 	& \equiv \Pm(\{\past \in E_{\beta}\}) \nonumber \\
			& = \P(\{\biinfinity = \past \future : \past \in E_{\beta}\})
\end{align}
to each equivalence class $E_{\beta}$. We say a process $\Process$ is
\emph{finitely characterized} if there are a finite number of positive probability
equivalence classes $E_1, \ldots , E_N$ that together comprise
a set of full measure: $\P(E_i) > 0$ for each $1\leq i \leq N$ and
$\sum_{i=1}^N \P(E_i) = 1$. For a finitely characterized process $\Process$ we
will also occasionally say, by a slight abuse of terminology, that
$\E^+ \equiv \{E_1, \ldots , E_N\}$ \emph{is} the set of equivalence classes
of pasts and ignore the remaining measure-zero subset of equivalence classes. 

Appendix \ref{appsec:Finitely_Characterized_Processes} shows that for any
finitely characterized process $\Process$, the transitions from the positive
probability equivalence classes $E_i \in \E^+$ all go to other positive
probability equivalence classes. That is, if $E_i \in \E^+$ then:
\begin{equation}
\delta_h(E_i, \ms) \in \E^+, \mbox{ for all } \ms \mbox{ with } \PB(\ms|E_i) > 0.
\label{eq:ConsistentTrans}
\end{equation}
As such, we define symbol-labeled transition matrices $T^{(\ms)}, \ms \in \XX$
between the equivalence classes $E_i \in \E^+$. A component $T_{ij}^{(\ms)}$ of
the matrix $T^{(\ms)}$ gives the probability that equivalence class $E_i$
transitions to equivalence class $E_j$ on symbol $\ms$:
\begin{align}
T^{(\ms)}_{ij} & = \PB(E_i  \goesonx E_j)
               \equiv  I(\ms,i,j) \cdot \PB(\ms|E_i) ~,
\end{align}
where $I(\ms,i,j)$ is the indicator function of the transition from $E_i$ to $E_j$ on symbol $\ms$:
\begin{align}
I(\ms,i,j) 
  & = \left\{ \begin{array}{ll}
  	1 & \mbox{if } \PB(\ms|E_i) > 0 \mbox{ and } \delta_h(E_i,\ms) = E_j, \\
  	0 & \mbox{otherwise.}
  \end{array} \right.
\end{align}
It follows from Equations (\ref{eq:Normalization}) and (\ref{eq:ConsistentTrans})
that the matrix $T \equiv \sum_{\ms \in \XX} T^{(\ms)}$ is stochastic.
(See also Claim \ref{cl:WellDefinednessOfHistoryMachine} in Appendix
\ref{appsec:Finitely_Characterized_Processes}.) 

\begin{Def}
\label{def:HeM}
Let $\Process = (\XXZ,\X,\P)$ be a finitely characterized, stationary, ergodic, finite-alphabet process. 
The history \eM\ $M_h(\Process)$ is defined as the triple $(\E^+, \XX, \{T^{(x)}\})$. 
\end{Def}
Note that $M_h$ is a valid HMM since $T$ is stochastic. 

\subsection{Examples}
\label{subsec:Examples}

In this section we present several examples of irreducible HMMs and the associated
\eMs\ for the processes that these HMMs generate. This should hopefully provide some 
useful intuition for the definitions. For the sake of brevity, descriptions of the history
\eM\ constructions in our examples will be less detailed than in the formal definition
given above, but the ideas should be clear. In all cases, 
the process alphabet is the binary alphabet $\XX = \{0,1\}$.

\begin{Exa}
Even Machine
\label{ex:Even} 
\end{Exa}

The first example we consider, shown in Figure \ref{fig:Even}, is the
generating HMM $M$ for the Even Process previously introduced in Section
\ref{subsec:HiddenMarkovModels}. It is easily seen that this HMM is both
irreducible and unifilar and, also, that it has probabilistically distinct
states. State $\cs_1$ can generate the symbol $0$, whereas state $\cs_2$
cannot. $M$ is therefore a generator \eM, and by Theorem
\ref{thm:GeneratorsAreHistoryMachines} below the history \eM\ $M_h$ for the
process $\Process$ that $M$ generates is isomorphic to $M$. The Fischer cover
for the sofic shift $\mbox{supp}(\Process)$ is also isomorphic to $M$,
if probabilities are removed from the edge labels in $M$.  

\begin{figure}[ht]
\includegraphics[scale=1.0]{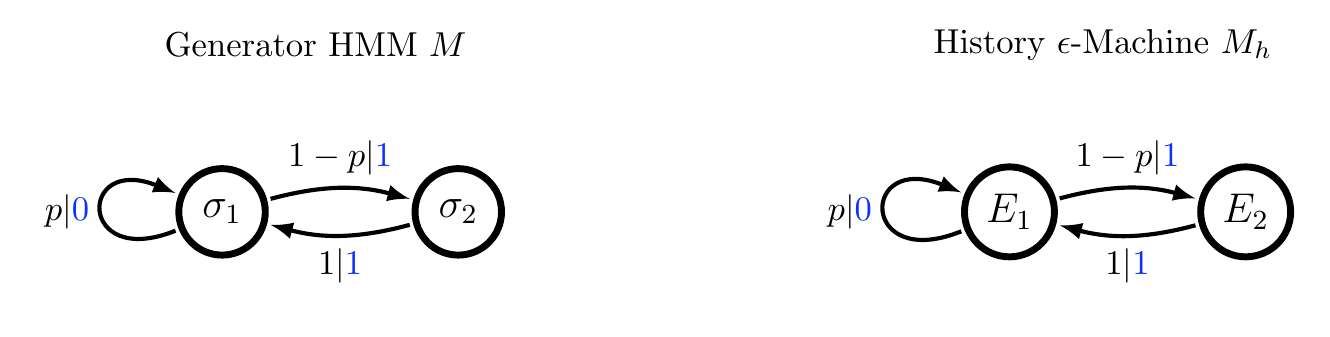}
\vspace{2 mm}
\caption{The Even Machine $M$ (left) and associated history \eM\ $M_h$ (right) 
for the process $\Process$ generated by $M$. $p \in (0,1)$ is a parameter.}
\label{fig:Even}
\end{figure}

More directly, the history \eM\ states for $\Process$ can be deduced by noting
that $0$ is a \emph{synchronizing word} for $M$ \cite{Trav11a}: It synchronizes
the observer to state $\cs_1$. Thus, for any nontrivial past $\past$ terminating in
$\ms_{-1} = 0$, the initial state $s_0$ must be $\cs_1$. By unifilarity, any
nontrivial past $\past$ terminating in a word of the form $01^n$ for some
$n \geq 0$ also uniquely determines the initial state $s_0$. For $n$ even, we
must have $s_0 = \cs_1$ and, for $n$ odd, we must have $s_0 = \cs_2$. Since
a.e. infinite past $\past$ generated by $M$ contains at least one $0$ and the
distributions over future sequences $\future$ are distinct for the two states
$\cs_1$ and $\cs_2$, the process $\Process$ is finitely characterized with
exactly two positive probability equivalence classes of infinite pasts: 
$E_1 = \{\past = \ldots 01^n: n \mbox{ is even } \}$ and
$E_2 = \{\past = \ldots 01^n: n \mbox{ is odd } \}$.
These correspond to the states $\cs_1$ and $\cs_2$ of $M$, respectively. More
generally, a similar argument holds for any \emph{exact} generator \eM. That
is, any generator \eM\ having a finite synchronizing word $w$ \cite{Trav11a}.
 
\begin{Exa}
Alternating Biased Coins Machine
\label{ex:ABC} 
\end{Exa}

Figure \ref{fig:ABC} depicts a generating HMM $M$ for the Alternating Biased
Coins (ABC) Process. This process may be thought of as being generated by
alternately flipping two coins with different biases $p \not= q$. The
phase---$p$-bias on odd flips or $p$-bias on even flips---is chosen uniformly
at random. $M$ is again, by inspection, a generator \eM: irreducible and unifilar with
probabilistically distinct states. Therefore, by Theorem
\ref{thm:GeneratorsAreHistoryMachines} below, the history \eM\ $M_h$ for the
process $\Process$ that $M$ generates is again isomorphic to $M$. However, the
Fischer cover for the sofic shift $\mbox{supp}(\Process)$ is not isomorphic to
$M$. The support of $\Process$ is the full shift $\XXZ$, so the Fischer cover
consists of a single state transitioning to itself on both symbols $0$ and $1$. 

\begin{figure}[h]
\includegraphics[scale=1.0]{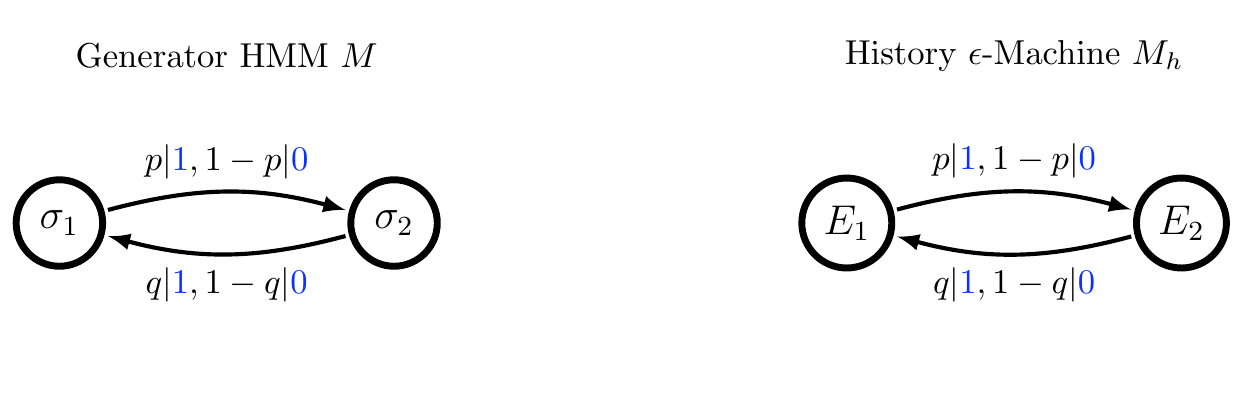}
\vspace{2 mm}
\caption{The Alternating Biased Coins (ABC) Machine $M$ (left) and associated
  history \eM\ $M_h$ (right) for the process $\Process$ generated by $M$.
  $p,q \in (0,1)$ are parameters, $p \not=q$.
  }
\label{fig:ABC}
\end{figure}

In this simple example, the history \eM\ states can also be deduced directly,
despite the fact that the generator $M$ does not have a synchronizing word.
If the initial state is $s_0 = \cs_1$, then by the strong law of large numbers
the limiting fraction of $1$s at odd time steps in finite-length past blocks
$\past^t$ converges a.s. to $q$. Whereas, if the initial state is $s_0 = \cs_2$,
then the limiting fraction of $1$s at odd time steps converges a.s. to $p$. Therefore, 
the initial state $s_0$ can be inferred a.s. from the complete past $\past$, so
the process $\Process$ is finitely characterized with two positive probability
equivalence classes of infinite pasts $E_1$ and $E_2$, corresponding to the two
states $\cs_1$ and $\cs_2$. Unlike the exact case, however, arguments like this
do not generalize as easily to other nonexact generator \eMs.

\begin{Exa}
Nonminimal Noisy Period-$2$ Machine
\label{ex:NP2} 
\end{Exa}

Figure \ref{fig:NP2} depicts a nonminimal generating HMM $M$ for the Noisy
Period-$2$ (NP2) Process $\Process$ in which $1$s alternate with random
symbols. $M$ is again unifilar, but it does not have probabilistically distinct
states and is, therefore, not a generator \eM. States $\cs_1$ and $\cs_3$ have
the same probability distribution over future output sequences as do states
$\cs_2$ and $\cs_4$.  

\begin{figure}[ht]
\includegraphics[scale=1.0]{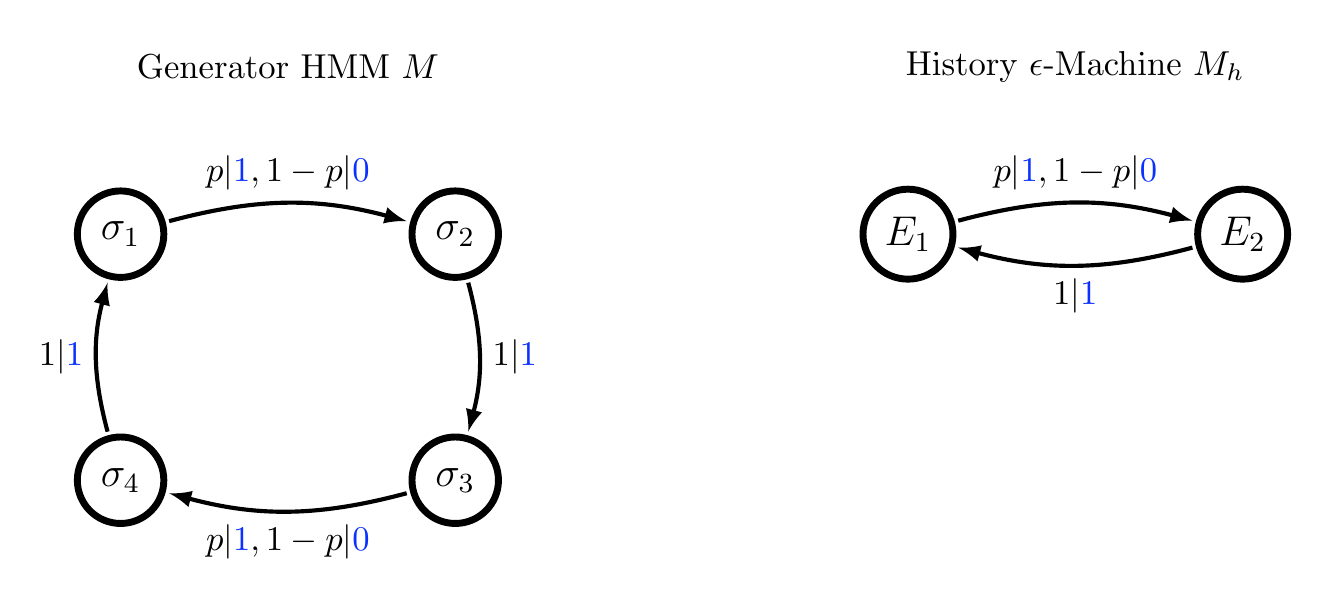}
\vspace{2 mm}
\caption{A nonminimal generating HMM $M$ for the Noisy Period-$2$ (NP2) Process
  (left), and the associated history \eM\ $M_h$ for this process (right).
  $p \in (0,1)$ is a parameter.
  }
\label{fig:NP2}
\end{figure}

There are two positive probability equivalence classes of pasts $\past$ for the
process $\Process$: Those containing $0$s at a subset of the odd time steps,
and those containing $0$s at a subset of the even time steps. Those with $0$s at
odd time steps induce distributions over future output equivalent to that
from states $\cs_2$ and $\cs_4$. While those with $0$s at even time steps
induce distributions over future output equivalent to that from states $\cs_1$
and $\cs_3$. Thus, the \eM\ for $\Process$ consists of just two states
$E_1 \sim \{\cs_1, \cs_3\}$ and $E_2 \sim \{\cs_2, \cs_4\}$. In general, for a
unifilar HMM without probabilistically distinct states the \eM\ is formed by
grouping together equivalent states in a similar fashion.

\begin{Exa}
Simple Nonunifilar Source 
\label{ex:SNS} 
\end{Exa}

Figure \ref{fig:SNS} depicts a generating HMM $M$ known as the Simple
Nonunifilar Source (SNS) \cite{Crut92c}. The output process $\Process$ generated by $M$
consists of long sequences of $1$s broken by isolated $0$s. As its name
indicates, $M$ is nonunifilar, so it is not an \eM. 

\begin{figure}[ht]
\includegraphics[scale=1.0]{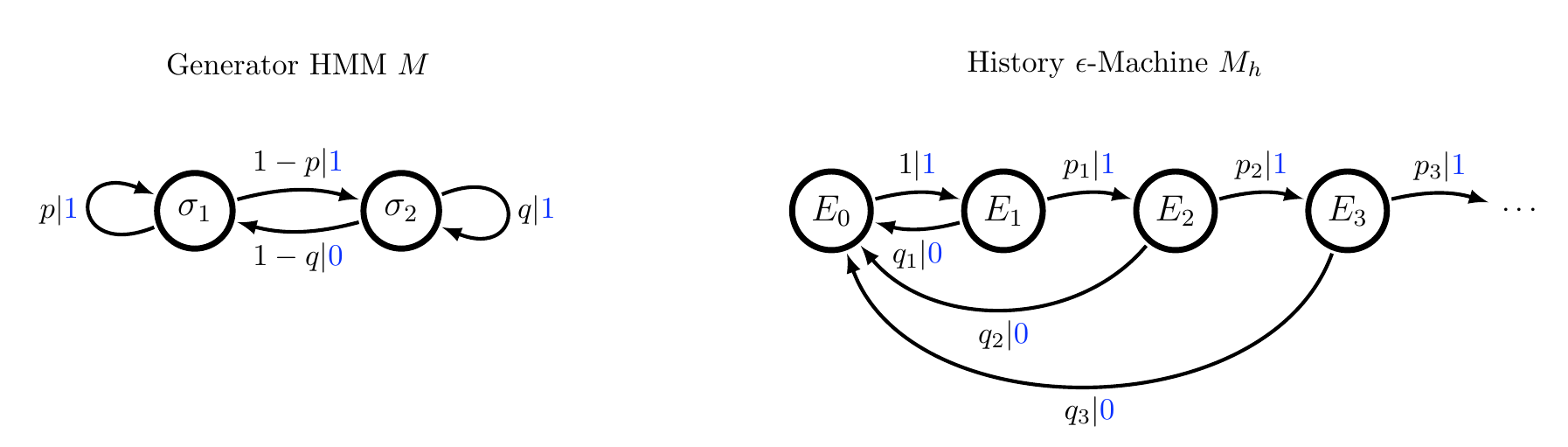}
\vspace{2 mm}
\caption{The Simple Nonunifilar Source (SNS) $M$ (left) and associated history
	\eM\ $M_h$ (right) for the process $\Process$ generated by $M$. In the 
	history \eM, $q_n + p_n = 1$ for each $n \in \N$ and $(q_n)_{n \in \N}$ is an
	increasing sequence defined by: 
	$q_n = (1-q) \cdot \left( (1-p) \sum_{m=0}^{n-1} p^m q^{n-1-m} \right) /
         \left( p^n + (1-p) \sum_{m=0}^{n-1} p^m q^{n-1-m} \right)$.
	}
\label{fig:SNS}
\end{figure}

Symbol $0$ is a synchronizing word for $M$, so all pasts $\past$ ending in
a $0$ induce the same probability distribution over future output sequences
$\future$: Namely, the distribution over futures given by starting $M$ in the
initial state $s_0 = \cs_1$. However, since $M$ is nonunifilar, an observer
does not remain synchronized after seeing a $0$. Any nontrivial past of the
form $\past = \ldots 01^n$ induces the same distribution over the initial state
$s_0$ as any other. However, for $n \geq 1$ there is some possibility of being
in both $\cs_1$ and $\cs_2$ at time $0$. A direct calculation shows that the
distributions over $s_0$ and, hence, the distributions over future output
sequences $\future$ are distinct for different values of $n$.
Thus, since a.e. past $\past$ contains at least one $0$, it follows that the
process $\Process$ has a countable collection of positive probability
equivalence classes of pasts, comprising a set of full measure:
$\{E_n: n = 0,1,2 \ldots \}$ where $E_n = \{ \past = \ldots 01^n\}$. 
This leads to a \emph{countable-state history \eM} $M_h$ as depicted
on the right of Figure \ref{fig:SNS}. We will not address countable-state
machines further here, as other technical issues arise in this case. 
Conceptually, however, it is similar to the finite-state case and
may be depicted graphically in an analogous fashion.  

\section{Equivalence}
\label{sec:Equivalence}

We will show that the two \eM\ definitions---history and generator---are
equivalent in the following sense: 
\begin{enumerate}
\item If $\Process$ is the process generated by a generator \eM\ $M_g$, then
	$\Process$ is finitely characterized and the history \eM\ $M_h(\Process)$
	is isomorphic to $M_g$ as a hidden Markov model. 
\item If $\Process$ is a finitely characterized, stationary, ergodic,
	finite-alphabet process, then the history \eM\ $M_h(\Process)$, 
	when considered as a hidden Markov model,
	is also a generator \eM. And, the process $\Process'$ generated by $M_h$
	is the same as the original process $\Process$ from which the 
	history machine was derived.  
\end{enumerate}      

That is, there is a $1-1$ correspondence between finite-state generator
\eMs\ and finite-state history \eMs. Every generator \eM\ is also a history
\eM, for the same process $\Process$ it generates. Every history
\eM\ is also a generator \eM, for the same process $\Process$ from which
it was derived.

\subsection{Generator \EMs\ are History \EMs} 
\label{subsec:Generator_eMachines_are_History_eMachines}

In this section we establish equivalence in the following direction:
\begin{The}
\label{thm:GeneratorsAreHistoryMachines}
If $\Process = (\XXZ, \X, \P)$ is the process generated by a generator \eM\ $M_g$,
then $\Process$ is finitely characterized and the history \eM\ $M_h(\Process)$
is isomorphic to $M_g$ as a hidden Markov model. 
\end{The}

The key ideas in proving this theorem come from the study of synchronization
to generator \eMs\ \cite{Trav11a,Trav11b}. In order to state these ideas precisely,
however, we first need to introduce some terminology.

Let $M_g$ be a generator \eM, and let $\Process = (\XXZ, \X, \P)$ and
$\Processt = (\XSZ, \XcrossS, \Pt)$ be the associated symbol and symbol-state
processes generated by $M_g$ as in Section \ref{subsec:HiddenMarkovModels}.
Further, let the random variables $\MS_t: \XSZ \rightarrow \XX$ and
$\CS_t:\XSZ \rightarrow \CSSet$ be the natural projections
$\MS_t(\biinfinity,\biinfinitystates) =\ms_t$ 
and $\CS_t(\biinfinity,\biinfinitystates) = \csr_t$, and let
$\Future^t = \MS_0 \ldots \MS_{t-1}$
and $\Past^t = \MS_{-t} \ldots \MS_{-1}$. 

The \emph{process language} $\L(\Process)$ is the set of words $w$ of positive
probability: $\L(\Process) = \{w \in \XX^*: \P(w) > 0 \}$. For a given word
$w \in \L(\Process)$, we define $\phi(w) = \Pt(\CSSet|w)$ to be an observer's
\emph{belief distribution} as to the machine's current state after observing
the word $w$. Specifically, for a length-$t$ word $w \in \L(\Process)$,
$\phi(w)$ is a probability distribution over the machine states
$\{\cs_1, \dots , \cs_N\}$ whose $i^{th}$ component is:
\begin{align}
\phi(w)_i 	& = \Pt(\CS_0 = \cs_i|\Past^t = w) \nonumber \\
		&= \Pt(\CS_0 = \cs_i,\Past^t = w)/\Pt(\Past^t = w) ~.
\end{align}
For a word $w \not\in \L(\Process)$ we will, by convention, take
$\phi(w) = \pi$. 

For any word $w$, $\csb(w)$ is defined to be the most likely machine 
state at the current time given that the word $w$ was just observed.
That is, $\csb(w) = \cs_{i^*}$, where $i^*$ is defined by the
relation $\phi(w)_{i^*} = \max_i$ $\phi(w)_i$. In the case of a tie, $i^*$ is
taken to be the lowest value of the index $i$ maximizing the quantity
$\phi(w)_i$. Also, $P(w)$ is defined to be the probability of the most
likely state after observing $w$:
\begin{align}
P(w) \equiv \phi(w)_{i^*}  ~.
\end{align}
And, $Q(w)$ is defined to be the combined
probability of all other states after observing $w$:
\begin{align}
Q(w) \equiv \sum_{i \not= i^*} \phi(w)_i = 1 - P(w) ~. 
\end{align}
So, for example, if $\phi(w) = (0.2,0.7,0.1)$ then $\csb(w) = \cs_2$, $P(w) = 0.7$, and $Q(w) = 0.3$. 

The most recent $t$ symbols are described by the  block random variable
$\Past^t$, and so we define the corresponding random variables
$\Phi_t = \phi(\Past^t)$, $\CSb_t = \csb(\Past^t)$, $P_t =  P(\Past^t)$, and
$Q_t = Q(\Past^t)$. Although the values depend only on the symbol sequence 
$\biinfinity$, formally we think of $\Phi_t$, $\CSb_t$, $P_t$, and $Q_t$ as defined on 
the cross product space $\XSZ$. Their realizations are denoted with
lowercase letters $\phi_t$, $\csrb_t$, $p_t$, and $q_t$, so that for a given
realization $(\biinfinity, \biinfinitystates) \in \XSZ$,
$\phi_t = \phi(\past^t)$, $\csrb_t = \csb(\past^t)$, $p_t = P(\past^t)$, and
$q_t = Q(\past^t)$. The primary result we use is the following exponential
decay bound on the quantity $Q_t$. 

\begin{Lem}
\label{lem:GeneratorConvergence}
For any generator \eM\ $M_g$ there exist constants $K > 0$ and $0 < \alpha < 1$
such that:
\begin{align} 
\Pt(Q_t > \alpha^t) \leq K \alpha^t, \mbox{ for all } t \in \N ~.
\end{align}
\end{Lem} 

\begin{proof}
This follows directly from the Exact Machine Synchronization Theorem of
\cite{Trav11a} and the Nonexact Machine Synchronization Theorem of
\cite{Trav11b} by stationarity.
(Note that the notation used there differs slightly from that here by a time shift
of length $t$. That is, $Q_t$ there refers to the observer's doubt in $\CS_t$
given $\Future^t$, instead of the observer's doubt in $\CS_0$ given $\Past^t$.
Also, $L$ is used as a time index rather than $t$ in those works.)
\end{proof}

Essentially, this lemma says that after observing a block of $t$ symbols it is
exponentially unlikely that an observer's doubt $Q_t$ in the machine state
will be more than exponentially small. Using the lemma we now
prove Theorem \ref{thm:GeneratorsAreHistoryMachines}.

\begin{proof} (Theorem \ref{thm:GeneratorsAreHistoryMachines})
Let $M_g$ be a generator \eM\ with state set
$\CSSet = \{\cs_1, \ldots , \cs_N \}$ and stationary distribution
$\pi = (\pi_1, \ldots , \pi_N)$. Let $\Process$ and $\Processt$ 
be the associated symbol and symbol-state processes generated by $M_g$. 
By Lemma \ref{lem:GeneratorConvergence} there exist constants $K>0$ and
$0<\alpha<1$ such that $\Pt(Q_t > \alpha^t) \leq K \alpha^t$, for all
$t \in \N$. Let us define sets:
\begin{align*}
& V_t = \{ (\biinfinity,\biinfinitystates): q_t \leq \alpha^t ~, \csr_0 = \csrb_t \} ~,\\
& V'_t = \{ (\biinfinity,\biinfinitystates): q_t \leq \alpha^t~, \csr_0 \not= \csrb_t \} ~,\\ 
& W_t = \{ (\biinfinity,\biinfinitystates): q_t > \alpha^t \} ~, \mbox{ and }\\
& U_t = W_t \cup V'_t ~.
\end{align*}
Then, we have:
\begin{align*}
\Pt(U_t) & = \Pt(V'_t) + \Pt(W_t) \\
  & \leq \alpha^t + K \alpha^t \\
  & = (K+1) \alpha^t ~.
\end{align*}
So:
\begin{align*}
\sum_{t = 1}^{\infty} \Pt(U_t)
  \leq \sum_{t = 1}^{\infty} (K+1)\alpha^t < \infty ~. 
\end{align*}
Hence, by the Borel-Cantelli Lemma, $\Pt(U_t \mbox{ occurs infinitely often})
= 0$. Or, equivalently, for $\Pt$ a.e. $(\biinfinity,\biinfinitystates)$
there exists $t_0 \in \N$ such that $(\biinfinity,\biinfinitystates) \in V_t$
for all $t \geq t_0$. Now, define:
\begin{align*}
C & = \{ (\biinfinity,\biinfinitystates): \mbox{ there exists } t_0 \in \N \mbox{ such that } 
(\biinfinity,\biinfinitystates) \in V_t \mbox{ for all } t \geq t_0 \} ~, \\
D_i & =  \{ (\biinfinity,\biinfinitystates): \csr_0 = \cs_i \} ~, \mbox{ and } \\ 
C_i & = C \cap D_i ~. 
\end{align*}
According to the above discussion $\Pt(C) = 1$ and, clearly, $\Pt(D_i) = \pi_i$.
Thus, $\Pt(C_i) = \Pt(C \cap D_i) = \pi_i$. Also, by the convention for
$\phi(w), w \not\in \L(\Process)$, we know that for every
$(\biinfinity, \biinfinitystates) \in C_i$, the corresponding symbol past
$\past$ is nontrivial. So, the conditional probabilities $\P(w|\past^t)$ are
well defined for each $t$. 

Now, given any $(\biinfinity, \biinfinitystates) \in C_i$ take $t_0$ sufficiently large so that for all
$t \geq t_0$, $(\biinfinity,\biinfinitystates) \in V_t$. Then, for $t \geq t_0$, $\csrb_t = \cs_i$ and 
$q_t \leq \alpha^t$. So, for any word $w \in \XX^*$ and any $t \geq t_0$, we have:
\begin{align}
|\P(w| & \past^t) - \PB_{\cs_i}(w)| \nonumber \\ 
& = \left| \Pt(\Future^{|w|} = w|\Past^t = \past^t) - \Pt(\Future^{|w|} = w|\CS_0 = \cs_i) \right| \nonumber \\
& \stackrel{(*)}{=}  \left|  \left\{  \sum_j \Pt(\Future^{|w|} = w|\CS_0 = \cs_j) \Pt(\CS_0 = \cs_j|\Past^t = \past^t) \right\}
- \Pt(\Future^{|w|} = w|\CS_0 = \cs_i) \right|
  \nonumber \\
& =  \left|  \left\{  \sum_{j \not= i} 
  \Pt(\Future^{|w|} = w|\CS_0 = \cs_j)
  \Pt(\CS_0 = \cs_j|\Past^t = \past^t) \right\}
  - \left(1 - \Pt(\CS_0 = \cs_i|\Past^t = \past^t) \right)
  \Pt(\Future^{|w|} = w|\CS_0 = \cs_i) \right|
  \nonumber \\
& \leq  \left\{  \sum_{j \not=i} \Pt(\Future^{|w|} = w|\CS_0 = \cs_j)
  \Pt(\CS_0 = \cs_j|\Past^t = \past^t) \right\}
  + \left(1 - \Pt(\CS_0 = \cs_i|\Past^t = \past^t) \right)
  \Pt(\Future^{|w|} = w|\CS_0 = \cs_i) \nonumber \\
& \leq  \left\{ \sum_{j \not=i} \Pt(\CS_0 = \cs_j|\Past^t = \past^t) \right\}
  + \left(1 - \Pt(\CS_0 = \cs_i|\Past^t = \past^t) \right)
  \nonumber \\
& = 2 q_t \nonumber \\
& \leq 2 \alpha^t ~. \nonumber 
\end{align} 
Step (*) follows from the fact that $\Past^m$ and $\Future^n$ are conditionally
independent given $\CS_0$ for any $m,n \in \N$, by construction of the measure
$\Pt$. Since $|\P(w| \past^t) - \PB_{\cs_i}(w)| \leq 2 \alpha^t$ for all
$t \geq t_0$, we know $\lim_{t \to \infty} \P(w|\past^t) = \PB_{\cs_i}(w)$ exists. 
Since this holds for all $w \in \XX^*$, we know $\past$ is regular and
$\PB(w|\past) = \PB_{\cs_i}(w)$ for all $w \in \XX^*$. 

Now, let us define equivalence classes $E_i$, $i = 1, \ldots ,N$, by:
\begin{align*}
E_i = \{ \past: \past \mbox{ is regular and } \PB(w|\past) = \PB_{\sigma_i}(w) \mbox{ for all } w \in \XX^* \} ~.  \\
\end{align*}
And, also, for each $i = 1, \ldots, N$ let:
\begin{align*}
\widetilde{E}_i = \{ (\biinfinity, \biinfinitystates): \past \in E_i \} ~.
\end{align*}
By results from Appendix \ref{appsec:Measurability_of_Equivalence_Classes} we know
that each equivalence class $E_i$ is measurable, so each set $\widetilde{E}_i$
is also measurable with $\Pt(\widetilde{E}_i) = \P(E_i)$. And, for each $i$,
$C_i \subseteq \widetilde{E}_i$, so
$\P(E_i) = \Pt(\widetilde{E}_i) \geq \Pt(C_i) = \pi_i$. Since
$\sum_{i=1}^N \pi_i = 1$ and the equivalence classes $E_i, i =
1,\ldots,N$, are all disjoint,
it follows that $\P(E_i) = \pi_i$ for each $i$, and 
$\sum_{i=1}^N \P(E_i) = \sum_{i=1}^N \pi_i = 1$. Hence, the process $\Process$
is finitely characterized with positive probability equivalences classes $\E^+ = \{E_1, \ldots, E_N\}$. 

Moreover, the equivalence classes $\{E_1, \ldots , E_N\}$---the history
\eM\ states---have a natural one-to-one correspondence with the states of the
generating \eM: $E_i \sim \cs_i, i = 1, \dots ,N$. It remains only to verify
that this bijection is also edge preserving and, thus, an isomorphism. 
Specifically, we must show that:
\begin{enumerate}
\item For each $i=1, \ldots ,N$ and $x \in \XX$, $\PB(x|E_i) = \PB_{\cs_i}(x)$, and
\item For all $i$ and $x$ with $\PB(x|E_i) = \PB_{\cs_i}(x) > 0$,
	$\delta_h(E_i,x) \cong \delta(\cs_i,x)$. That is, if $\delta_h(E_i,x) = E_j$
	and $\delta(\cs_i,x) = \cs_{j'}$, then $j = j'$. 
\end{enumerate}  

Point 1 follows directly from the definition of $E_i$. To show Point 2, take
any $i$ and $x$ with $\PB(x|E_i) = \PB_{\cs_i}(x) > 0$ and let $\delta_h(E_i,x) = E_j$
and $\delta(\cs_i,x) = \cs_{j'}$. Then, for any word $w \in \XX^*$, we have:
\begin{itemize}
\item[(i)] $\PB(\ms w|E_i) = \PB_{\cs_i}(\ms w)$, by definition of the equivalence
	class $E_i$,
\item[(ii)] $\PB(\ms w|E_i) = \PB(x|E_i) \cdot \PB(w|E_j)$, by Claim
	\ref{cl:ConsistencyOfOneStepTransitions} in Appendix
	\ref{appsec:Probabilistic_Consistency_of_Equivalence_Class_Transitions},
	and
\item[(iii)] $\PB_{\cs_i}(\ms w) = \PB_{\cs_i}(x) \cdot \PB_{\cs_{j'}}(w)$,
	by Equation (\ref{eq:StateToWordProb}) applied to a unifilar HMM.
\end{itemize}
Since $\PB(x|E_i) = \PB_{\cs_i}(x) > 0$, it follows that
$\PB(w|E_j) = \PB_{\cs_{j'}}(w)$. Since this holds for all $w \in \XX^*$ and
the states of the generator are probabilistically distinct, by assumption,
it follows that $j = j'$. 

\end{proof}

\begin{Cor}
\label{cor:GeneratorsAreUnique}
Generator \eMs\ are unique: Two generator \eMs\ $M_{g_1}$ and $M_{g_2}$
that generate the same process $\Process$ are isomorphic.
\end{Cor}

\begin{proof}
By Theorem \ref{thm:GeneratorsAreHistoryMachines} the two generator \eMs\ are
both isomorphic to the process's history \eM\ $M_h(\Process)$ and, hence, 
isomorphic to each other.
\end{proof}

\begin{Rem}
Unlike history \eMs\ that are unique by construction, generator \eMs\ are not
by definition unique. And, it is not a priori clear that they must be. Indeed, general
HMMs are not unique. There are infinitely many nonisomorphic HMMs for
any given process $\Process$ generated by some HMM. Moreover, if
either the unifilarity or probabilistically distinct states condition is removed
from the definition of generator \eMs, then uniqueness no longer holds. It is only
when both of these properties are required together that one obtains uniqueness. 
\end{Rem}

\subsection{History \EMs\ are Generator \EMs} 
\label{subsec:History_eMachines_are_Generator_eMachines}

In this section we establish equivalence in the reverse direction:

\begin{The}
\label{thm:HistoryMachinesAreGenerators}
If $\Process$ is a finitely characterized, stationary, ergodic, finite-alphabet
process, then the history \eM\ $M_h(\Process)$, when considered as a hidden
Markov model, is also a generator \eM. And, the process $\Process'$ generated
by $M_h$ is the same as the original process $\Process$ from which the 
history machine was derived.    
\end{The} 

Note that by Claim \ref{cl:WellDefinednessOfHistoryMachine} in Appendix
\ref{appsec:Finitely_Characterized_Processes} we know that for any finitely
characterized, stationary, ergodic, finite-alphabet process the history \eM\
$M_h(\Process) = (\E^+, \XX, \{ T^{(x)} \})$ is a valid hidden Markov model.
So, we need only show that this HMM has the three properties of a generator
\eM---strongly connected graph, unifilar transitions, and probabilistically
distinct states---and that the process $\Process'$ generated by this HMM is
the same as $\Process$. Unifilarity is immediate from the construction, but the
other claims take more work and require several lemmas to establish.
Throughout $\mu = (\mu_1, \ldots , \mu_N) \equiv (\P(E_1), \ldots , \P(E_N) )$,
where $\E^+ = \{E_1, \ldots , E_n\}$ is the set of positive probability equivalence classes 
for the process $\Process$. 

\begin{Lem}
\label{lem:EquivalenceClassPiStationarity}
The distribution $\mu$ over equivalence-class states is stationary for the
transition matrix $T = \sum_{x \in \XX} T^{(x)}$. That is, 
for any $1 \leq j \leq N$, $\mu_j = \sum_{i=1}^N \mu_i \cdot T_{ij}$.
\end{Lem} 

\begin{proof}
This follows directly from Claim \ref{cl:EquivalenceClassStationarity} in Appendix
\ref{appsec:Finitely_Characterized_Processes} and the definition of the $T^{(x)}$ matrices. 
\end{proof}

\begin{Lem}
\label{lem:DisjointStronglyConnectedComponents}
The graph $G$ associated with the HMM $M_h = (\E^+, \XX, \{T^{(x)}\})$ consists entirely of 
disjoint strongly connected components. Each connected component of $G$ is strongly connected. 
\end{Lem}

\begin{proof}
It is equivalent to show that the graphical representation of the associated
Markov chain with state set $\E^+$ and transition matrix $T$ consists entirely
of disjoint strongly connected components. But this follows directly from the
existence of a stationary distribution $\mu$ with $\mu_i = \P(E_i) > 0$ for
all $i$ \cite{Levi06}.  
\end{proof} 

\begin{Lem}
\label{lem:ConditionalWordProbs} 
For any $E_i\in \E^+$ and $w \in \XX^*$, $\PB(w|E_i) = \PB_{E_i}(w)$, 
where $\PB_{E_i}(w) \equiv \norm{e_i T^{(w)}}_1$ is the probability of generating the word
$w$ starting in state $E_i$ of the HMM $M_h = (\E^+, \XX, \{ T^{(x)} \})$ as defined
in Section \ref{subsec:HiddenMarkovModels}. 
\end{Lem}

\begin{proof}

By construction $M_h$ is a unifilar HMM, and its transition function $\delta$,
as defined in Section \ref{subsec:Generator_eMachines}, is the same as the transition
function $\delta_h$ between equivalence classes of histories as defined 
in Equation (\ref{eq:DeltaFunctionHistoryMachineDef}). Moreover, we have by construction
that for each $x \in \XX$ and state $E_i$, $\PB_{E_i}(x) = \PB(x|E_i)$. The lemma follows essentially
from these facts. We consider separately the two cases $\PB(w|E_i) > 0$ and $\PB(w|E_i) = 0$. 

\begin{itemize}

\item Case (i) - $\PB(w|E_i) > 0$. 
Let $w = w_0 \ldots w_{l-1}$ be a word of length $l \geq 1$ with $\PB(w|E_i) > 0$. 
By Claim \ref{cl:ConistencyOfChainedTransitions} in Appendix 
\ref{appsec:Probabilistic_Consistency_of_Equivalence_Class_Transitions}
and the ensuing remark we know that the equivalence classes 
$s_0 = E_i$, $s_1 = \delta_h(s_0,w_0), \ldots , s_l = \delta_h(s_{l-1}, w_{l-1})$ are well defined and: 
\begin{align*}
\PB(w|E_i) = \prod_{t=0}^{l-1} \PB(w_t|s_t) ~.
\end{align*} 
Since $\delta_h \cong \delta$ we see that there is an allowed state path $s$ in
the HMM $M_h$---namely, $s = s_1, \ldots ,s_l$---such that the word $w$ can be
generated following $s$ from the initial state $E_i$. It follows that
$\PB_{E_i}(w) > 0$ and given by Equation (\ref{eq:WordProbsUnifilarHMM}):
\begin{align*}
\PB_{E_i}(w) = \prod_{t=0}^{l-1} \PB_{s_t}(w_t) = \prod_{t=0}^{l-1} \PB(w_t|s_t) ~.
\end{align*}

\item Case (ii) - $\PB(w|E_i) = 0$. 
Let $w = w_0 \ldots w_{l-1}$ be a word of length $l \geq 1$ with $\PB(w|E_i) = 0$. 
For $0 \leq m \leq l-1$, define $w^m = w_0 \ldots w_{m-1}$ ($w^0$ is the null word $\lambda$). 
Take the largest integer $m \in \{0,\ldots,l-1\}$ such that $\PB(w^m|E_i) > 0$. 
By convention we take $\PB(\lambda|E_i) = 1$ for all $i$, 
so there is always some such $m$. A similar analysis to above then shows that 
the equivalence classes $s_0, \ldots , s_m$ defined by $s_0 = E_i$, 
$s_{t+1} = \delta_h(s_t, w_t)$ are well defined and:
\begin{align*}
\PB(w^{m+1}|E_i) = \prod_{t=0}^{m} \PB(w_t|s_t) = \PB_{E_i}(w^{m+1}). 
\end{align*}
By our choice of $m$, $\PB(w^{m+1}|E_i) = 0$, so $\PB_{E_i}(w^{m+1}) = 0$ as well. 
It follows that $\PB_{E_i}(w) = 0$, since $w^{m+1}$ is a prefix of $w$. 

\end{itemize}
\end{proof}

\begin{Lem}
\label{lem:WordProbs} 
For any $w \in \XX^*$, $\P(w) = \norm{ \mu T^{(w)} }_1$. 
\end{Lem}

\begin{proof} 
Let $E_{i,w} \equiv \{\biinfinity: \future^{|w|} = w, \past \in E_i \}$. Claim
\ref{cl:ProbabilisticConsistency} of
Appendix \ref{appsec:Probabilistic_Consistency_of_Equivalence_Class_Transitions}
shows that each $E_{i,w}$ is an $\X$-measurable set with
$\P(E_{i,w}) =  \P(E_i) \cdot \PB(w|E_i)$. Since the $E_i$s are disjoint sets
with probabilities summing to 1, it follows that
$\P(w) = \sum_{i=1}^N \P(E_{i,w})$ for each $w \in \XX^*$. Thus, applying
Lemma \ref{lem:ConditionalWordProbs}, for any $w \in \XX^*$ we have:
\begin{align*}
\P(w) 	& = \sum_{i=1}^N \P(E_{i,w}) \\
		& = \sum_{i=1}^N \P(E_i) \cdot \PB(w|E_i) \\
		& = \sum_{i=1}^N \mu_i \norm{ e_i T^{(w)} }_1 \\
		& = \norm{ \mu T^{(w)} }_1 ~.
\end{align*}
\end{proof}

\begin{proof}(Theorem \ref{thm:HistoryMachinesAreGenerators})
\begin{enumerate}
\item \emph{Unifilarity}: As mentioned above, this is immediate from the history
	\eM\ construction. 
\item \emph{Probabilistically Distinct States}: Take any $i$ and $j$ with
$i \not=j$. By construction of the equivalence classes there exists 
	some word $w \in \XX^*$ such that $\PB(w|E_i) \not= \PB(w|E_j)$. 
	But by Lemma \ref{lem:ConditionalWordProbs},
	$\PB(w|E_i) = \PB_{E_i}(w)$ and $\PB(w|E_j) = \PB_{E_j}(w)$. Hence, 
	$\PB_{E_i}(w) \not= \PB_{E_j}(w)$, so the states 
	$E_i$ and $E_j$ of the HMM $M_h = (\E^+, \XX, \{ T^{(x)} \})$ are
	probabilistically distinct. Since this holds for all $i \not= j$, 
	$M_h$ has probabilistically distinct states. 
\item \emph{Strongly Connected Graph}: By Lemma
	\ref{lem:DisjointStronglyConnectedComponents}, we know the graph $G$
	associated with the HMM $M_h$ consists of one or more connected components
	$C_1, \ldots , C_n$, each of which is strongly connected. Assume that there
	is more than one of these strongly connected components: $n \geq 2$. By
	Points 1 and 2 above we know that each component $C_k$ defines a generator
	\eM. If two of these components $C_k$ and $C_j$ were isomorphic via a
	function $f: C_k$ states $\rightarrow C_j$ states, then for states
	$E_i \in C_k$ and $E_l \in C_j$ with $f(E_i) = E_l$, we would have 
	$\PB_{E_i}(w) = \PB_{E_l}(w)$ for all $w \in \XX^*$. By Lemma
	\ref{lem:ConditionalWordProbs}, however, this implies
	$\PB(w|E_i) = \PB(w|E_l)$ for all $w \in \XX^*$ as well, which contradicts
	the fact that $E_i$ and $E_l$ are distinct equivalence classes. Hence,
	no two of the components $C_k, k = 1, \ldots , n$, can be isomorphic.
	By Corollary \ref{cor:GeneratorsAreUnique}, this implies that the stationary
	processes $\Process^k, k = 1, \ldots , n$, generated by each of the
	generator \eM\ components are all distinct. But, by a block diagonalization
	argument, it follows from Lemma \ref{lem:WordProbs} that 
	$\Process = \sum_{k=1}^n \mu^k \cdot \Process^k$, where
	$\mu^k = \sum_{\{i: E_i \in C_k\}} \mu_i$. 
	That is, for any word $w \in \XX^*$, we have:
\begin{align*}
\P(w)		& = \sum_{k=1}^n \mu^k \cdot \P^k(w) \\
     		& = \sum_{k=1}^n \mu^k \cdot \norm{\rho^k T^{k,(w)}}_1 ~,
\end{align*}
	where $\rho^k$ and $T^{k,(w)}$ are, respectively, the stationary state
	distribution and $w$-transition matrix for the generator \eM\ of component
	$C_k$. Since the $\Process^k$s are all distinct, this implies that the
	process $\Process$ cannot be ergodic, which is a contradiction. Hence,
	there can only be one strongly connected component $C_1$---the whole
	graph is strongly connected. 	
\item \emph{Equivalence of $\Process$ and $\Process'$}: Since the graph of
	the HMM $M_h = (\E^+, \XX, \{ T^{(x)} \})$ is strongly connected there is a
	unique stationary distribution $\pi$ over the states satisfying
	$\pi = \pi T$. However, we already know the distribution $\mu$ is stationary.
	Hence, $\pi = \mu$. By definition, the word probabilities $\P'(w)$ for
	the process $\Process'$ generated by this HMM are
	$\P'(w) = \norm{ \pi T^{(w)} }_1, w \in \XX^*$. 
	But, by Lemma \ref{lem:WordProbs}, we have also
	$\P(w) = \norm{ \mu T^{(w)} }_1 = \norm{ \pi T^{(w)} }_1$
	for each $w \in \XX^*$. Hence, $\P(w) = \P'(w)$ for all $w \in \XX^*$,
	so $\Process$ and $\Process'$ are the same process. 
\end{enumerate}
\end{proof}

\section{Conclusion}
\label{sec:Conclusion}

We have demonstrated the equivalence of finite-state history and generator \eMs.
This is not a new idea. However, a formal treatment was absent until quite
recently. While the rigorous development of \eMs\ in \cite{Lohr10a} also 
implies equivalence, the proofs given here, especially for Theorem
\ref{thm:GeneratorsAreHistoryMachines}, are more direct
and provide improved intuition.

The key step in proving the equivalence, at least the new approach used for 
Theorem \ref{thm:GeneratorsAreHistoryMachines}, comes directly from recent
bounds on synchronization rates for finite-state generator \eMs. To generalize
the equivalence to larger model classes, such as machines with a countably
infinite number of states, it therefore seems reasonable that one should
first deduce and apply similar synchronization results for countable-state
generators. Unfortunately, for countable-state generators synchronization can
be much more difficult and exponential decay rates as in Lemma
\ref{lem:GeneratorConvergence} no longer always hold. Thus, it is unclear
whether equivalence in the countable-state case always holds either. Though, 
the results in \cite{Lohr10a} do indicate equivalence holds for countable-state
machines if the entropy in the stationary distribution $H[\pi]$ is finite, which it often is. 

\section*{Acknowledgments}
This work was partially support by ARO award W911NF-12-1-0234-0.
NT was partially supported by an NSF VIGRE fellowship.

\appendix

\section{Regular Pasts and Trivial Pasts} 
\label{appsec:Regular_Pasts_and_Trivial_Pasts}

We establish that the set of trivial pasts $\TT$ is a null set and the set
of regular pasts $\RR$ has full measure. Throughout this section
$\Process = (\XXZ, \X , \P)$ is a stationary, ergodic process over a finite
alphabet $\XX$, and $(\XXm, \Xm , \Pm)$ is the corresponding probability
space over past sequences $\past$. Other notation is used as in Section
\ref{sec:Definitions}.  

\begin{Cla}
\label{cl:TrivialPasts} 
$\Pm$ a.e. $\past$ is nontrivial. That is, $\TT$ is an $\Xm$ measurable set with $\Pm(\TT) = 0$. 
\end{Cla} 

\begin{proof}
For any fixed $t$, $\TT_t \equiv \{ \past : \P(\past^t) = 0\}$ is $\Xm$ measurable, since it is $\X_t^-$ measurable, and
$\Pm(\TT_t) = 0$. Hence, $\TT = \bigcup_{t=1}^{\infty} \TT_t^-$ is also $\Xm$ measurable with $\Pm(\TT) = 0$. 
\end{proof} 

\begin{Cla}
\label{cl:wRegularPastsProb1}
For any $w \in \XX^*$, $\Pm$ a.e. $\past$ is $w$-regular. That is:
\begin{align*}
\RR_w \equiv \{ \past: \P(\past^t) > 0,
	\mbox{ for all } t \mbox{ and }
	\lim_{t \to \infty} \P(w|\past^t) \mbox{ exists} \}
\end{align*}
is an $\Xm$ measurable set with $\Pm(\RR_w) = 1$.
\end{Cla} 

\begin{proof}
Fix $w \in \XX^*$. Let $Y_{w,t} : \XX^- \rightarrow \R$ be defined by:
\begin{align*}
Y_{w,t}(\past)
  & = \left\{ \begin{array}{ll}
  	\P(w|\past^t) & \mbox{if } \P(\past^t) > 0, \\
  	0 & \mbox{otherwise.}
  \end{array} \right.
\end{align*} 
Then, the sequence $(Y_{w,t})$ is a martingale with respect to the filtration
$(\X_t^-)$ and $\Ex(Y_{w,t}) \leq 1$ for all $t$. Hence, by the Martingale
Converge Theorem $Y_{w,t} \stackrel{a.s.}{\longrightarrow} Y_w$ for some
$\Xm$ measurable random variable $Y_w$. In particular,
$\lim_{t \to \infty} Y_{w,t}(\past)$ exists for $\Pm$ a.e. $\past$. 

Let $\widehat{\RR}_{w} \equiv \{\past: \lim_{t \to \infty} Y_{w,t}(\past) \mbox{ exists} \}$.
Then, as just shown, $\widehat{\RR}_{w}$ is $\Xm$ measurable with $\Pm(\widehat{\RR}_{w}) = 1$, 
and from Claim \ref{cl:TrivialPasts}, we know $\TT$ is $\Xm$ measurable with $\Pm(\TT) = 0$. 
Hence, $\RR_w = \widehat{\RR}_{w} \cap \TT^c$ is also $\Xm$ measurable with $\Pm(\RR_w) = 1$. 
\end{proof}

\begin{Cla}
\label{cl:RegularPastsProb1}
$\Pm$ a.e. $\past$ is regular. That is, $\RR$ is an $\Xm$ measurable set with $\Pm(\RR)$ = 1. 
\end{Cla}

\begin{proof}
$\RR = \bigcap_{w \in \XX^*} \RR_w $.
By Claim \ref{cl:wRegularPastsProb1}, each $\RR_w$ is $\Xm$ measurable with
$\Pm(\RR_w) = 1$. Since there are only countably many finite length words
$w \in \XX^*$, it follows that $\RR$ is also $\Pm$ measurable with
$\Pm(\RR)$ = 1.
\end{proof}

\section{Well Definedness of Equivalence Class Transitions}
\label{appsec:Well_Definedness_of_Equivalence_Class_Transitions}

We establish that the equivalence-class-to-equivalence-class transitions are
well defined and normalized for the equivalence classes $ E_{\beta} \in \E$.
Throughout this section $\Process = (\XXZ, \X , \P)$ is a stationary, ergodic
process over a finite alphabet $\XX$ and $(\XXm, \Xm , \Pm)$ is the
corresponding probability space over past sequences $\past$. Other notation
is used as in Section \ref{sec:Definitions}. Recall that, by definition, for any
regular past $\past$, $\P(\past^t) > 0$ for each $t \in \N$. This fact is
used implicitly in the proofs of the following claims several times to ensure
that various quantities are well defined.  

\begin{Cla}
\label{cl:AddWord_PositiveProbAtFiniteLengths}
For any regular past $\past \in \XXm$ and word $w \in \XX^*$ with
$\PB(w|\past) > 0:$
\begin{enumerate}
\item[(i)] $\P(\past^t w) > 0$ for each $t \in \N$ and
\item[(ii)] $\P(w|\past^t) > 0$ for each $t \in \N$. 
\end{enumerate}
\end{Cla}

\begin{proof}
Fix any regular past $\past \in \XXm$ and word
$w \in \XX^*$ with $\PB(w|\past) > 0$. 
Assume there exists $t \in \N$ such that $\P(\past^t w) = 0$. Then
$\P(\past^n w) = 0$ for all $n \geq t$ 
and, thus, $\P(w|\past^n) = \P(\past^n w)/\P(\past^n) = 0$ for all $n \geq t$ as well. 
Taking the limit gives $\PB(w|\past) = \lim_{n \to \infty} \P(w|\past^n) = 0$, which is a
contradiction. Hence, we must have $\P(\past^t w) > 0$ for each $t$, proving
(i). (ii) follows since $\P(w|\past^t) = \P(\past^t w)/\P(\past^t)$ is greater than
zero as long as $\P(\past^t w) > 0$.
\end{proof}

\begin{Cla}
\label{cl:AddSymbol_ToRegularPast_IsRegular}
For any regular past $\past \in \XXm$ and any symbol $\ms \in \XX$
with $\PB(\ms|\past) > 0$, the past $\past \ms$ is regular. 
\end{Cla}

\begin{proof}
Fix any regular past $\past \in \XXm$ and symbol $\ms \in \XX$ with
$\PB(\ms|\past) > 0$. By Claim \ref{cl:AddWord_PositiveProbAtFiniteLengths},
$\P(\past^t \ms)$ and $\P(\ms|\past^t)$ are both nonzero for each $t \in \N$.
Thus, the past $\past \ms$ is nontrivial, and the conditional probability
$\P(w|\past^t \ms)$ is well defined for each $w \in \XX^*, t \in \N$ and
given by:
\begin{align*}
\P(w|\past^t \ms) = \frac{\P(\ms w|\past^t)}{\P(\ms|\past^t)} ~. 
\end{align*}
Taking the limit gives:
\begin{align*}
\lim_{t \to \infty} \P(w|(\past \ms)^t) 
& = \lim_{t \to \infty} \P(w|\past^t \ms) \\
& = \lim_{t \to \infty} \frac{\P(\ms w|\past^t)}{\P(\ms |\past^t)} \\
& = \frac{\lim_{t \to \infty} \P(\ms w|\past^t)}{\lim_{t \to
\infty} \P(\ms|\past^t)} \\
& = \frac{\PB(\ms w|\past)}{\PB(\ms |\past)} ~.
\end{align*}
In particular, $\lim_{t \to \infty} \P(w|(\past \ms)^t) = \PB(\ms
w|\past)/\PB(\ms |\past)$ exists. 
Since this holds for all $w \in \XX^*$, the past $\past \ms$ is regular. 

\end{proof}

\begin{Cla}
\label{cl:UnifilarityOfEquivalenceClassTransitions} 
If $\past$ and $\past'$ are two regular pasts in the same equivalence class
$E_{\beta} \in \E$ then, for any symbol $\ms \in \XX$ with
$\PB(\ms |E_{\beta}) > 0$, the regular pasts $\past \ms$ and
$\past' \ms$ are also
in the same equivalence class. 
\end{Cla}

\begin{proof}
Let $E_{\beta} \in \E$ and fix any $\past, \past' \in E_{\beta}$ and
$\ms \in \XX$ with $\PB(\ms |E_{\beta}) = \PB(\ms |\past) =
\PB(\ms|\past') > 0$.
By Claim \ref{cl:AddSymbol_ToRegularPast_IsRegular}, $\past \ms$
and $\past' \ms$
are both regular. And, just as in the proof of Claim
\ref{cl:AddSymbol_ToRegularPast_IsRegular}, for any $w \in \XX^*$ we have:
\begin{align*}
\PB(w|\past \ms) = \lim_{t \to \infty} \P(w|(\past \ms)^t) =
\frac{\PB(\ms w|\past)}{\PB(\ms|\past)} =
\frac{\PB(\ms w|E_{\beta})}{\PB(\ms |E_{\beta})} ~.
\end{align*} 
Also, similarly, for any $w \in \XX^*$:
\begin{align*}
\PB(w|\past' \ms) = \lim_{t \to \infty} \P(w|(\past' \ms)^t) =
\frac{\PB(\ms w|\past')}{\PB(\ms |\past')} =
\frac{\PB(\ms w|E_{\beta})}{\PB(\ms |E_{\beta})} ~.
\end{align*} 
Since this holds for all $w \in \XX^*$, it follows that $\past
\ms$ and $\past' \ms$ are both in the same equivalence class. 
\end{proof}

\begin{Cla}
\label{cl:Normalization}
For any equivalence class $E_{\beta}$, $\sum_{\ms \in \XX} \PB(\ms |E_{\beta}) = 1$. 
\end{Cla}

\begin{proof}
Fix $\past \in E_{\beta}$. Then: 
\begin{align*}
\sum_{x \in \XX} \PB(x|E_{\beta}) 
& = \sum_{x \in \XX} \PB(x|\past) \\
& = \sum_{x \in \XX} \lim_{t \to \infty} \P(x|\past^t) \\
& =   \lim_{t \to \infty} \sum_{x \in \XX} \P(x|\past^t) \\
& = \lim_{t \to \infty} 1 \\
& = 1.
\end{align*}
\end{proof}

\section{Measurability of Equivalence Classes}
\label{appsec:Measurability_of_Equivalence_Classes}

We establish that the equivalence classes $E_{\beta}, \beta \in B$, are
measurable sets. Throughout this section $\Process = (\XXZ, \X , \P)$ is
a stationary, ergodic process over a finite alphabet $\XX$ and
$(\XXm, \Xm , \Pm)$ is the corresponding probability space over past
sequences $\past$. Other notation is used as in Section \ref{sec:Definitions}.  

\begin{Cla}
\label{cl:MeasurabilityOfAwp} 
Let $\AA_{w,p} \equiv \{\past: \P(\past^t) > 0, \mbox{ for all } t \mbox{~and~} \lim_{t \to \infty} \P(w|\past^t) = p\}$. 
Then $\AA_{w,p}$ is $\Xm$ measurable for each $w \in \XX^*$ and $p \in [0,1]$. 
\end{Cla}

\begin{proof}
We proceed in steps through a series of intermediate sets.
\begin{itemize}
\item Let $\AA_{w,p,\epsilon,t}^+ \equiv \{\past: \P(\past^t) > 0,~ \P(w|\past^t) \leq p + \epsilon \}$ 
	and $\AA_{w,p,\epsilon,t}^- \equiv \{\past: \P(\past^t) > 0,~ \P(w|\past^t) \geq p - \epsilon \}$.\\
	$\AA_{w,p,\epsilon,t}^+$ and $\AA_{w,p,\epsilon,t}^-$ are both $\Xm$ measurable, since they are both $\Xm_t$ measurable. 
\item Let $\AA_{w,p,\epsilon}^+ \equiv \bigcup_{n=1}^{\infty} \bigcap_{t=n}^{\infty} \AA_{w,p,\epsilon,t}^+ 
		= \{\past: \P(\past^t) > 0, \forall t \mbox{ and } \exists n \in \N \mbox{ such that } \P(w|\past^t) \leq p + \epsilon, \mbox{ for } t \geq n \}$,
	and $\AA_{w,p,\epsilon}^- \equiv \bigcup_{n=1}^{\infty} \bigcap_{t=n}^{\infty} \AA_{w,p,\epsilon,t}^- 
		= \{\past: \P(\past^t) > 0, \forall t \mbox{ and } \exists n \in \N \mbox{ such that } \P(w|\past^t) \geq p - \epsilon, \mbox{ for } t \geq n \}$.
	Then $\AA_{w,p,\epsilon}^+$ and $\AA_{w,p,\epsilon}^-$ are each $\Xm$ measurable since they are countable unions of countable 
	intersections of $\Xm$ measurable sets. 
\item Let $\AA_{w,p,\epsilon} \equiv \AA_{w,p,\epsilon}^+ \cap \AA_{w,p,\epsilon}^- 
	= \left\{\past: \P(\past^t) > 0, \forall t \mbox{ and } \exists n \in \N \mbox{ such that } \left| \P(w|\past^t) - p \right| \leq \epsilon, \mbox{ for } t \geq n \right\}$.
	$\AA_{w,p,\epsilon}$ is $\Xm$ measurable since it is the intersection of two $\Xm$ measurable sets. 
\item Finally, note that $\AA_{w,p} = \bigcap_{m=1}^{\infty} \AA_{w,p,\epsilon_m}$ where $\epsilon_m = 1/m$. 
	And, hence, $\AA_{w,p}$ is $\Xm$ measurable as it is a countable intersection of $\Xm$ measurable sets.  
\end{itemize}
\end{proof}

\begin{Cla}
\label{cl:MeasurabilityOfEquivalenceClasses} 
Any equivalence class $E_{\beta} \in \E$ is an $\Xm$ measurable set.
\end{Cla}

\begin{proof}
Fix any equivalence class $E_{\beta} \in \E$ and, for $w \in \XX^*$, let
$p_w = \PB(w|E_{\beta})$. 
By definition $E_{\beta} = \bigcap_{w \in \XX^*} \AA_{w,p_w}$ and, by Claim 
\ref{cl:MeasurabilityOfAwp}, each $\AA_{w,p_w}$ is $\Xm$ is measurable. 
Thus, since there are only countably many finite length words $w \in \XX^*$, 
$E_{\beta}$ must also be $\Xm$ measurable. 
\end{proof}

\section{Probabilistic Consistency of Equivalence Class Transitions}
\label{appsec:Probabilistic_Consistency_of_Equivalence_Class_Transitions}

We establish that the probability of word generation from each equivalence
class is consistent in the sense of Claims
\ref{cl:ConistencyOfChainedTransitions} and \ref{cl:ProbabilisticConsistency}.
Claim \ref{cl:ProbabilisticConsistency} is used in the proof of Claim
\ref{cl:EquivalenceClassStationarity} in Appendix
\ref{appsec:Finitely_Characterized_Processes}, and Claim
\ref{cl:ConistencyOfChainedTransitions} is used in the proof of Theorem
\ref{thm:HistoryMachinesAreGenerators}. Throughout this section we assume
$\Process = (\XXZ, \X, \P)$ is a stationary, ergodic process over a finite
alphabet $\XX$ and denote the corresponding probability space over past
sequences as $(\XXm, \Xm, \Pm)$, with other notation is as in Section
\ref{sec:Definitions}. 
We define also the \emph{history $\sigma$-algebra} $\H$ for a process
$\Process = (\XXZ, \X, \P)$ as the $\sigma$-algebra generated
by cylinder sets of all finite length histories. That is, 
\begin{align*}
\H = \sigma \left( \bigcup_{t=1}^{\infty} \H_t \right) \mbox{ ~where~ }
\H_t = \sigma \left( \{ A_{w,-|w|} : |w| = t \} \right),
\end{align*}
with $A_{w,t} = \{\biinfinity: \ms_{t} \ldots \ms_{t+|w|-1} = w\}$ as in
Section \ref{sec:Definitions}. 
$\H$ is the projection onto $\XX^{\Z}$ of the $\sigma$-algebra $\X^-$ on the space $\XX^-$.

\begin{Cla}
\label{cl:ExtensionConsistencyBound}
For any $E_{\beta} \in \E$ and $w,v \in \XX^*$, $ \PB(wv|E_{\beta}) \leq \PB(w|E_{\beta}) $.
\end{Cla}

\begin{proof}
Fix $\past \in E_{\beta}$. Since $\P(wv|\past^t) \leq \P(w|\past^t)$ for each $t$: 
\begin{align*}
\PB(wv|E_{\beta}) = \PB(wv|\past) = \lim_{t \to \infty} \P(wv|\past^t) \leq \lim_{t \to \infty} \P(w|\past^t) = \PB(w|\past) = \PB(w|E_{\beta}).
\end{align*} 
\end{proof}

\begin{Cla}
\label{cl:ConsistencyOfOneStepTransitions}
Let $E_{\beta} \in \E$, $\ms \in \XX$ with $\PB(\ms|E_{\beta}) > 0$, and let
$E_{\alpha} = \delta_h(E_{\beta},\ms)$. Then,
$\PB(\ms w|E_{\beta}) = \PB(\ms |E_{\beta}) \cdot \PB(w|E_{\alpha})$
for any word $w \in \XX^*$.
\end{Cla}

\begin{proof}
Fix $\past \in E_{\beta}$. Then $\past \ms \in E_{\alpha}$ is regular, so
$\P(\past^t \ms ) > 0$ for all $t$ and we have:
\begin{align*}
\PB(\ms w|E_{\beta}) 	& = \PB(\ms w|\past) \\ 
				& = \lim_{t \to \infty} \P(\ms w|\past^t) \\
				& = \lim_{t \to \infty} \P(\ms |\past^t) \cdot \P(w|\past^t \ms) \\
				& = \lim_{t \to \infty} \P(\ms |\past^t) \cdot
				\lim_{t \to \infty} \P(w|\past^t \ms ) \\
				& = \PB(\ms |\past) \cdot \PB(w|\past \ms ) \\
				& = \PB(\ms |E_{\beta}) \cdot \PB(w|E_{\alpha}) ~.
\end{align*}
\end{proof}

\begin{Cla}
\label{cl:ConistencyOfChainedTransitions}
Let $w = w_0 \ldots w_{l-1} \in \XX^*$ be a word of length $l \geq 1$,
and let $w^m = w_0 \ldots w_{m-1}$ for $0 \leq m \leq l$. 
Assume that $\PB(w^{l-1}|E_{\beta}) > 0$ for some $E_{\beta} \in
\E$. Then the equivalence classes $s_t$, $0 \leq t \leq l-1$,
defined inductively by the relations $s_0 = E_{\beta}$ and
$s_t = \delta_h(s_{t-1}, w_{t-1})$ for $1 \leq t \leq l - 1$,
are well defined. That is, $\PB(w_{t-1}|s_{t-1}) > 0$ for each
$1 \leq t \leq l-1$. Further, the probability $\PB(w|E_{\beta})$ may be
expressed as: 
\begin{align*}
\PB(w|E_{\beta}) = \prod_{t = 0}^{l-1} \PB(w_t|s_t).
\end{align*}
\noindent
In the above, $w^0 = \lambda$ is the null word and, for any equivalence class $E_{\beta}$, $\PB(\lambda|E_{\beta}) \equiv 1$. 
\end{Cla}

\begin{proof}
For $|w| = 1$ the statement is immediate and, for $|w| = 2$, it reduces to
Claim \ref{cl:ConsistencyOfOneStepTransitions}. For $|w| \geq 3$, it can
proved by induction on the length of $w$ using Claim
\ref{cl:ConsistencyOfOneStepTransitions} and the consistency 
bound provided by Claim \ref{cl:ExtensionConsistencyBound} which 
guarantees that $\PB(w_0|E_{\beta}) > 0$ if $\PB(w^{l-1}|E_{\beta}) > 0$. 
\end{proof} 

\begin{Rem}
If $\PB(w|E_{\beta}) > 0$, then by Claim \ref{cl:ExtensionConsistencyBound} we know 
$\PB(w^{l-1}|E_{\beta}) > 0$, so the formula above holds for any word $w$ with $\PB(w|E_{\beta}) > 0$. 
Moreover, in this case, $\PB(w_{l-1}|s_{l-1})$ must be nonzero in order to ensure $\PB(w|E_{\beta})$ is 
nonzero. Thus, the equivalence class $s_l = \delta_h(s_{l-1},w_{l-1})$ is also well defined. 
\end{Rem} 

The following theorem from \cite[Chapter 4, Theorem 5.7]{Durr96a} is needed in
the proof of Claim \ref{cl:VersionOfConditionalExpectation}. It is an
application of the Martingale Convergence Theorem. 

\begin{The}
\label{thm:DurrettMartingale}
Let $(\Omega, \FF, \P)$ be a probability space, and let
$\FF_1 \subseteq \FF_2 \subseteq \FF_3 \ldots$ be an increasing sequence of
$\sigma$-algebras on $\Omega$ with
$\FF_{\infty} = \sigma(\bigcup_{n=1}^{\infty} \FF_n) \subseteq \FF$.
Suppose $X: \Omega \rightarrow \R$ is an $\FF$-measurable random variable
(with $\Ex|X| < \infty$). Then, for (any versions of) the conditional
expectations $\Ex(X|\FF_n)$ and $\Ex(X|\FF_{\infty})$, we have:
\begin{align*}
\Ex(X|\FF_n) \rightarrow \Ex(X|\FF_{\infty}) \mbox{ a.s. and in $L^1$. }
\end{align*}
\end{The}

\begin{Cla}
\label{cl:VersionOfConditionalExpectation}
For any $w \in \XX^*$, $\PB_w(\biinfinity)$ is (a version of) the conditional expectation $\Ex\left(\indicator_{A_{w,0}}|\H\right)(\biinfinity)$,
where $\PB_w: \XXZ \rightarrow [0,1]$ is defined by: 
\begin{align*}
\PB_w(\biinfinity)
  & = \left\{ \begin{array}{ll}
  	\PB(w|\past) & \mbox{if } \past \mbox{ is regular,
		where $\biinfinity = \past \future$, } \\
  	0 & \mbox{otherwise.}
  \end{array} \right.
\end{align*}
\end{Cla}

\begin{proof}
Fix $w \in \XX^*$, and let $\Ex_w$ be any fixed version of the conditional expectation $\Ex\left(\indicator_{A_{w,0}}|\H\right)$. 
Since the function $\PB_{w,t}: \XXZ \rightarrow [0,1]$ defined by:
\begin{align*}
\PB_{w,t}(\biinfinity)
  & = \left\{ \begin{array}{ll}
  	\P(w|\past^t) & \mbox{if } \P(\past^t) > 0, \\
  	0 & \mbox{otherwise}
  \end{array} \right.
\end{align*}
is a version of the conditional expectation $\Ex(\indicator_{A_{w,0}}|\H_t)$,
Theorem \ref{thm:DurrettMartingale} implies that 
$\PB_{w,t}(\biinfinity) \rightarrow \Ex_w(\biinfinity)$ for $\P$
a.e. $\biinfinity$. Now, define:
\begin{align*}
V_w & = \{\biinfinity: \PB_{w,t}(\biinfinity) \rightarrow \Ex_w(\biinfinity) \}, \\
W_w & = \{\biinfinity \in V_w: \past \mbox{ is regular} \}~.
\end{align*}
By the above $\P(V_w) = 1$ and, by Claim \ref{cl:RegularPastsProb1}, the 
regular pasts have probability 1. Hence, $\P(W_w) = 1$. 

However, for each $\biinfinity \in W_w$ we have:
\begin{align*}
\PB_w(\biinfinity) = \PB(w|\past) = \Ex_w(\biinfinity) ~. 
\end{align*}
Thus, $\PB_w(\biinfinity) = \Ex_w(\biinfinity)$ for $\P$ a.e. $\biinfinity$.
So, for any $\H$-measurable set $H$, $\int_H \PB_w ~ d\P= \int_H \Ex_w ~ d\P$.
Furthermore, $\PB_w$ is $\H$-measurable since 
$\PB_{w,t} \stackrel{a.s.}{\longrightarrow} \PB_w$ and each $\PB_{w,t}$ is
$\H$-measurable. It follows that $\PB_w(\biinfinity)$ is a version of the
conditional expectation $\Ex\left(\indicator_{A_{w,0}}|\H\right)$. 
\end{proof}

\begin{Cla}
\label{cl:ProbabilisticConsistency} 
For any equivalence class $E_{\beta} \in \E$ and word $w \in \XX^*$, the set 
$E_{\beta,w} \equiv \{\biinfinity: \past \in E_{\beta}, \future^{|w|} = w\}$ is
$\X$-measurable with $\P(E_{\beta,w}) = \P(E_{\beta}) \cdot \PB(w|E_{\beta})$. 
\end{Cla} 

\begin{proof}
Let $\widehat{E}_{\beta} = \{ \biinfinity: \past \in E_{\beta} \}$. 
Then $\widehat{E}_{\beta}$ and $A_{w,0}$ are both $\X$-measurable, 
so their intersection $E_{\beta,w}$ is as well. And, we have:
\begin{align*}
\P(E_{\beta,w}) 	& = \int_{\widehat{E}_{\beta}} \indicator_{A_{w,0}}(\biinfinity) ~ d\P \\
			& \stackrel{(a)}{=} \int_{\widehat{E}_{\beta}} \Ex(\indicator_{A_{w,0}}|\H) (\biinfinity) ~ d\P \\
			& \stackrel{(b)}{=} \int_{\widehat{E}_{\beta}} \PB_w (\biinfinity) ~ d\P \\
			& = \int_{\widehat{E}_{\beta}} \PB(w|E_{\beta}) ~ d\P \\
			& = \P(E_{\beta}) \cdot \PB(w|E_{\beta}) ~,
\end{align*} 
where (a) follows from the fact that $\widehat{E}_{\beta}$ is $\H$-measurable
and (b) follows from Claim \ref{cl:VersionOfConditionalExpectation}. 
\end{proof}
	 
\section{Finitely Characterized Processes}
\label{appsec:Finitely_Characterized_Processes}

We establish several results concerning finitely characterized processes. 
In particular, we show (Claim \ref{cl:WellDefinednessOfHistoryMachine})
that the history \eM\ $M_h(\Process)$ is, in fact, a well defined HMM.
Throughout, we assume $\Process = (\XXZ, \X, \P)$ is a stationary, ergodic,
finitely characterized process over a finite alphabet $\XX$ and
denote the corresponding probability space over past sequences as
$(\XXm, \Xm, \Pm)$. The set of positive probability equivalences is denoted
$\E^+ = \{E_1, \ldots, E_N\}$ and the set of all equivalence classes
as $\E = \{E_{\beta}, \beta \in B\}$. For equivalence classes $E_{\beta}, E_{\alpha} \in \E$ and symbol $x \in \XX$,
$I(x, \alpha, \beta)$ is the indicator of the transition from class $E_{\alpha}$ to class $E_{\beta}$ on symbol $x$. 
\begin{align*}
I(x,\alpha,\beta) 
  & = \left\{ \begin{array}{ll}
  	1 & \mbox{if } \PB(x|E_{\alpha}) > 0 \mbox{ and } \delta_h(E_{\alpha},x) = E_{\beta}, \\
  	0 & \mbox{otherwise.}
  \end{array} \right.
\end{align*} 
Finally, the symbol-labeled transition matrices $T^{(x)}, x \in \XX$ between
equivalence classes $E_1, \ldots , E_N$ are defined by
$T^{(x)}_{ij} = \PB(x|E_i) \cdot I(x,i,j)$. The overall transition matrix
between these equivalence classes is denoted by $T$, $T = \sum_ {x \in \XX} T^{(x)}$. 

\begin{Cla}
\label{cl:EquivalenceClassStationarity}
For any equivalence class $E_{\beta} \in \E$:
\begin{align*}
\P(E_{\beta}) = \sum_{i=1}^N  \sum_{x \in \XX}  \P(E_i) \cdot \PB(x|E_i) \cdot I(x,i,\beta) ~. 
\end{align*}
\end{Cla}

\begin{proof}
We have:
\begin{align*}
\P(E_{\beta}) 	& \equiv \P(\{\biinfinity: \past \in E_{\beta} \}) \\
			& \stackrel{(a)}{=}  \P(\{\biinfinity: \past\ms_0 \in E_{\beta} \}) \\
			& \stackrel{(b)}{=} \sum_{i=1}^N \P(\{\biinfinity: \past\ms_0 \in E_{\beta}, \past \in E_i \}) \\
			& = \sum_{i=1}^N \sum_{x \in \XX} \P(\{\biinfinity: \past\ms_0 \in E_{\beta}, \past \in E_i,\ms_0 = x \}) \\
			& = \sum_{i=1}^N \sum_{x \in \XX} \P(\{\biinfinity: \past \in E_i,\ms_0 = x \}) \cdot I(x,i,\beta) \\
			& = \sum_{i=1}^N \sum_{x \in \XX} \P(E_{i,x}) \cdot I(x,i,\beta) \\
			& \stackrel{(c)}{=} \sum_{i=1}^N \sum_{x \in \XX}  \P(E_i) \cdot \PB(x|E_i) \cdot I(x,i,\beta) ~,
\end{align*}
where (a) follows from stationarity, (b) from the fact that $\sum_{i=1}^N \P(E_i) = 1$,
and (c) from Claim \ref{cl:ProbabilisticConsistency}. 
\end{proof}

\begin{Cla}
\label{cl:InternalTransitions} 
For any $E_i \in \E^+$ and symbol $x$ with $\PB(x|E_i) > 0$, $\delta_h(E_i,x) \in \E^+$. 
\end{Cla}

\begin{proof}
Fix $E_i \in \E^+$ and $x \in \XX$ with $\PB(x|E_i) > 0$. By Claim \ref{cl:EquivalenceClassStationarity}, 
$\P( \delta_h(E_i,x) ) \geq \P(E_i) \cdot \PB(x|E_i) > 0$. Hence, $\delta_h(E_i,x) \in \E^+$. 
\end{proof}

\begin{Cla}
\label{cl:WellDefinednessOfHistoryMachine}
The transition matrix $T = \sum_{x \in X} T^{(x)}$ is stochastic: $\sum_{j=1}^N T_{ij} = 1$, for each $1 \leq i \leq N$. 
Hence, the HMM $M_h(\Process) = (\E^+, \XX, \{T^{(x)}\})$ is well defined. 
\end{Cla}

\begin{proof}
This follows directly from Claims \ref{cl:Normalization} and \ref{cl:InternalTransitions}. 
\end{proof}

\bibliography{ref,chaos}

\end{document}